\documentclass[11pt]{article}
\usepackage[utf8]{inputenc}
\usepackage{tikz}
\usepackage{float}
\usepackage{amsmath, amssymb, amsthm}
\usepackage[margin=1in]{geometry}
\usetikzlibrary{calc}
\usepackage[backend=bibtex, style=numeric]{biblatex}
\usepackage[mathlines]{lineno}
\usepackage{authblk}
\usepackage{hyperref}

\newtheorem{definition}{Definition}
\newtheorem{assumption}{Assumption}
\newtheorem{lemma}{Lemma}

\newtheorem{theorem}{Theorem}

\title{A Coupled Conforming--Nonconforming Galerkin Method for Poisson's Equation on Curved Domains}

\author[1]{Qingguang Guan
\thanks{First author. Email:\href{mailto:qingguang.guan@usm.edu}{qingguang.guan@usm.edu}} 
}
\author[2]{Wenju Zhao 
	\thanks{Corresponding author. Email:\href{mailto:zhaowj@sdu.edu.cn}{zhaowj@sdu.edu.cn}.
	Wenju Zhao was partially supported by National Key R\& D Program of China (No. 2023YFA10089033), Natural Science Foundation of Shandong Province (No. ZR2023ZD38), National Natural Science Foundation of China (No.12131014).
	} 
	}

\affil[1]{School of Mathematics and Natural Sciences, University of Southern Mississippi, Hattiesburg, MS 39406, USA}
\affil[2]{School of Mathematics, Shandong University, Jinan, Shandong 250100, China}

\date{}

\begin{document}
	
	\maketitle
	
	\begingroup
	\renewcommand{\thefootnote}{}
	\footnotetext{Submitted to {\em Numerical Methods for Partial Differential Equations}: March 8, 2026. Accepted: June 18, 2026.}
	\endgroup
	
	
\begin{abstract}
A coupled conforming--nonconforming Galerkin method is proposed for Poisson's equation on two-dimensional curved domains. The method applies a weak Galerkin discretization only on a thin boundary layer of curvilinear elements near the curved boundary, while using a standard continuous Galerkin discretization in the polygonal interior. In this way, geometric flexibility is retained where it is needed most, and the number of nonconforming degrees of freedom is significantly reduced.
A key ingredient is a mixed interpolation--projection operator on curvilinear weak Galerkin elements, combining \(L^2\) edge projections with conforming nodal interpolation on the interface side to ensure compatibility with the continuous Galerkin trace. Based on this construction, we prove an optimal a priori error estimate of order \(O(h^k)\) in the energy norm  {under the basic geometric assumptions of the method, and an optimal \(L^2(\Omega)\) estimate of order \(O(h^{k+1})\) under the additional elliptic dual regularity assumption}. Numerical experiments confirm the theoretical rates and demonstrate substantial savings in degrees of freedom compared with a fully weak Galerkin discretization.
\end{abstract}


	
	\section{Introduction}\label{sec:intro}
	
	Accurate numerical simulation of partial differential equations on \emph{curved domains} requires controlling
	two coupled error mechanisms: the approximation of the solution space and the approximation of the geometry.
	Even for the Poisson problem
	\[
	-\Delta u=f\quad\text{in }\Omega,\qquad u=0\quad\text{on }\partial\Omega,
	\]
	a naive polygonal approximation of a smooth boundary can degrade convergence, especially for higher-order
	methods, unless geometry errors are treated carefully.
	
	\medskip
	\noindent 
	A classical remedy is to employ \emph{boundary-fitted curved elements}, such as isoparametric finite elements. However, maintaining the optimal convergence rate of these methods is contingent upon the strict invertibility of the Jacobian matrix associated with the diffeomorphic mapping from the reference element to the physical element. Furthermore, optimal error bounds require bounds on the higher-order derivatives of this mapping, which inherently rely on sufficient smoothness of the curved boundary. Constructing such valid mappings is a highly non-trivial task, see \cite{lenoir1986optimal,zlamal1973curved,ciarlet2002finite}.
	An alternative is to decouple the mesh from the boundary through \emph{unfitted/embedded} techniques,
	including fictitious domain, immersed/cut finite element methods, which enforce boundary conditions weakly and add stabilization on cut elements, see \cite{de2023stability,he2023error,burman2025cut}. Such methods avoid curved meshes and are good at dealing with moving interfaces/boundaries, but they are considerably more difficult to implement and use robustly in 3D than in 2D. In 3D a cut cell becomes an arbitrary polyhedron (possibly with dozens of faces). Accurate volume integration over these shapes requires sophisticated algorithms. 
	A third successful direction is the development of \emph{polygonal} discretizations that
	naturally accommodate complex boundaries and interfaces, such as discontinuous Galerkin (DG) \cite{cangiani2014hp, cangiani2016hp, antonietti2016review, bertoluzza2021polygonal, bertoluzza2021p}, hybridizable-DG
	(HDG) \cite{cockburn2014solving, solano2019high,gurkan2019extended,kirby2012cg}, virtual element methods (VEM) \cite{da2024c,beirao2020polynomial,da2019virtual}, weak Galerkin (WG) methods \cite{guan2020weak,guan2018weak,wang2014weak,mu2015weak} and many others.

	\medskip
	\noindent 
	Nonconforming methods have proved particularly attractive for complex geometries because their
	degrees of freedom are naturally associated with element boundaries and can be adapted to curved edges/faces.
	For instance, in the WG framework, recent work \cite{Guan2023} has introduced high-order WG formulations on \emph{curvilinear polytopal}
	meshes with Lipschitz continuous curved edges/faces, enabling convergence rates that are independent of the
	geometry approximation under suitable regularity assumptions.

	\medskip
	\noindent 
	While DG/HDG/nonconforming-VEM/WG methods offer strong geometric flexibility, they also introduce extra interface unknowns
	(traces, numerical fluxes, or boundary variables) throughout the domain. In many applications, however, the
	\emph{primary difficulty is the curved boundary}, whereas the interior of the domain is geometrically simple
	and does not benefit from a fully nonconforming treatment. This observation motivates a \emph{coupled}
	conforming--nonconforming strategy: use a geometry-robust nonconforming method only in a thin layer adjacent
	to $\partial\Omega$, and retain a standard conforming method in the interior. Due to its simplicity and flexibility, which facilitates analysis, we select the Weak Galerkin method as the nonconforming component and the Continuous Galerkin (CG) method as the conforming component to develop a coupled conforming-nonconforming Galerkin method. Other nonconforming methods remain valid candidates.
	This philosophy is consistent with a long line of \emph{hybrid couplings} in the Galerkin literature.
	Coupled continuous/discontinuous Galerkin schemes were introduced to deploy discontinuous formulations only
	where additional flexibility is required (e.g., nonsmooth solution features), while preserving
	the efficiency of continuous elements elsewhere, see \cite{DawsonProft2002}. In the WG setting, coupled WG--CG formulations have been
	studied for \emph{physical} interface problems, see \cite{Zhai20} for Stokes--Darcy, where different discretizations are
	naturally used in different subregions and are coupled across an interface. 
	The present work differs in that the interface is not fixed; rather, it separates a curved boundary layer—which thins as the mesh is refined—from a polygonal interior. Additionally, our method solves a single-valued PDE on a single domain, necessitating fundamentally different techniques and error analysis.
		
	\medskip
	\noindent 
	We propose and analyze a coupled conforming--nonconforming Galerkin method for Poisson's equation on
	two-dimensional $C^{0,1}$ domains (Lipschitz-continuous boundary). The method employs:
	\begin{itemize}
		\item a \textbf{boundary-layer weak Galerkin discretization} on curvilinear elements whose outer side lies on
		$\partial\Omega$, designed to treat curved boundaries directly;
		\item a \textbf{standard continuous Galerkin discretization} on a triangulation of the polygonal interior, reducing degrees of freedom significantly;
		\item a \textbf{strong coupling} across the artificial interface $\Gamma$ between the two regions, so that
		the global trial space remains conforming across $\Gamma$ while allowing nonconforming traces on the outer
		and outer--inner sides of boundary-layer elements;
		\item a new \textbf{mixed interpolation--projection operator} on curvilinear WG elements, using $L^2$ edge
		projections on the curved/outer sides and conforming nodal interpolation on the inner side to match the CG
		trace on $\Gamma$.
	\end{itemize}
	On shape-regular curvilinear elements (star-shaped with a uniform chunkiness parameter), we establish
	stability and approximation estimates for the mixed interpolation--projection operator, and derive optimal
	order error estimates: $\mathcal O(h^k)$ in the natural energy norm  {under the geometric assumptions stated below, and $\mathcal O(h^{k+1})$ in $L^2(\Omega)$ under the additional dual elliptic regularity assumption introduced in Section~\ref{subsec:L2}}. From a practical standpoint, the coupled scheme reduces the global number of degrees of freedom compared to a nonconforming discretization while preserving optimal convergence rate, since the nonconforming unknowns are restricted to the boundary layer.
	
	\medskip
	\noindent 
	The paper is organized as follows. Section~\ref{sec:preliminaries} collects geometric assumptions and auxiliary approximation tools on
	curvilinear elements and introduces the interpolation--projection operator.
	Section~\ref{sec:WG-CG scheme} presents the coupled WG--CG scheme and proves well-posedness.
	Section~\ref{sec:error estimates} establishes energy and $L^2$ error estimates.
	Section~\ref{sec:numerics} reports numerical experiments on the unit disk that confirm the theoretical
	convergence rates and illustrate the reduction in degrees of freedom.
	
	\section{Preliminaries and Auxiliary Lemmas}
	\label{sec:preliminaries}
	
	While this paper focuses exclusively on two-dimensional domains with curved boundaries, the analysis and methodology can be readily extended to three dimensions. In this section, we collect the necessary geometric definitions, trace inequalities, and polynomial approximation results required to analyze the interpolation-projection weak Galerkin operator on curvilinear elements.

	\subsection{Domain Properties and Shape Regularity}
	Let $\Omega \subset \mathbb{R}^2$ be a bounded domain with a $C^{0,1}$ curved boundary.  {This Lipschitz regularity is the basic geometric assumption used to define the boundary layer and invoke the local trace and approximation results on curvilinear elements.} The mesh is constructed in two steps: (1) create a ``boundary layer,'' as shown in the left panel of Figure~\ref{fig:mesh-element}; every element in this layer can be represented by the enlarged element as in the right panel; (2) after the boundary layer is created, the remaining interior region is a polygon, which is then triangulated. 
		
	\medskip
	\noindent 
	Let $D$ denote a curvilinear element in the boundary layer, whose curved side is a subset of $\partial \Omega$. This element $D$ has four sides: one outer side, two outer--inner sides, and one inner side. On $D$, we use the $L^2$ 
	projection from $L^2(e)$ onto $\mathbb{P}_k|_{e}$ (the trace sense as in \cite{Guan2023}) when $e$ is an outer side or an outer--inner side, and we use the conforming interpolation onto $\mathbb{P}_k(e)$ 
	when $e$ is the inner side. Let $T$ denote a triangular element in the interior.
	With this design, we apply the weak Galerkin method only in the boundary layer and the CG  method in the interior. This hybrid approach combines the advantages of both methods: it provides a better treatment of the curved boundary while reducing the number of unknowns in the interior region.
	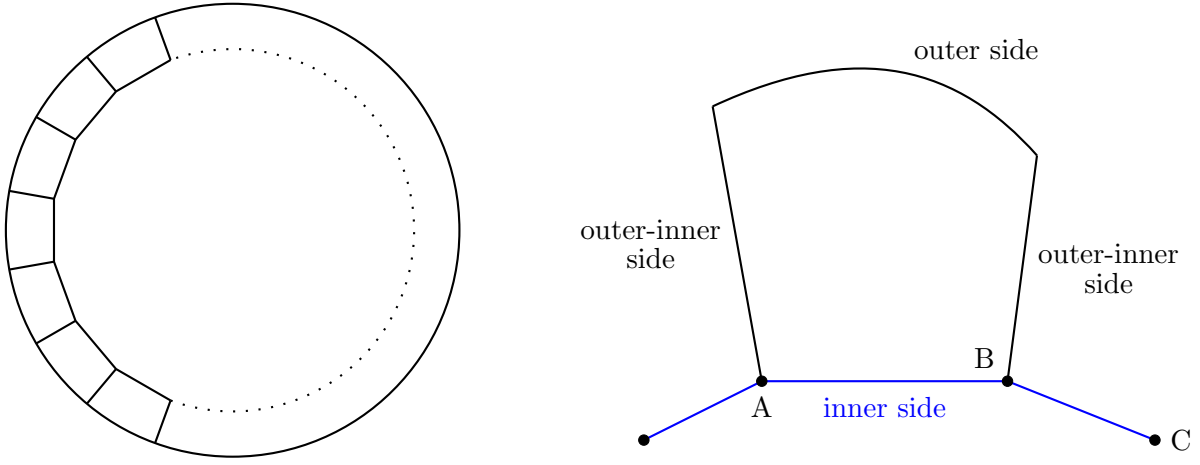
\begin{figure}[h!]
		\begin{center}
			\begin{tikzpicture}
				
				\begin{scope}[scale=1.0]
					\def\R{3}         
					\def\factor{0.8}  
					\def\r{\factor*\R} 
					
					\draw[thick] (0,0) circle (\R);
					
					\draw[thick] (110:\r) 
					\foreach \a in {130, 150, 170, 190, 210, 230, 250} {
						-- (\a:\r)
					};
					
					\foreach \a in {110, 130, 150, 170, 190, 210, 230, 250} {
						\draw[thick] (\a:\r) -- (\a:\R);
					}
					
					\draw[thick, loosely dotted] (250:\r) arc (250:470:\r);
				\end{scope}
				
				\begin{scope}[shift={(7, -2)}, scale=1.3]
					\coordinate (A) at (0, 0);
					\coordinate (B) at (2.5, 0);
					\coordinate (C) at (4, -0.6);
					\coordinate (A_ext) at (-1.2, -0.6);
					\coordinate (TopLeft) at (-0.5, 2.8);
					\coordinate (TopRight) at (2.8, 2.3);
					
					\draw[thick,blue] (A) -- (B) node[midway, below=2pt] {inner side};
					\draw[thick,blue] (A_ext) -- (A);
					\draw[thick,blue] (B) -- (C);
					\draw[thick] (A) -- (TopLeft) node[midway, left=2pt, align=center] {outer-inner\\[-0.5ex]side};
					\draw[thick] (B) -- (TopRight) node[midway, right=2pt, align=center] {outer-inner\\[-0.5ex]side};
					
					\draw[thick] (TopLeft) .. controls (1, 3.5) and (2, 3.2) .. (TopRight) 
					node[midway, above right=1pt] {outer side};
					
					\filldraw (A) circle (1.5pt) node[below=2pt] {A};
					\filldraw (B) circle (1.5pt) node[above left=1pt] {B};
					\filldraw (C) circle (1.5pt) node[right=2pt] {C};
					\filldraw (A_ext) circle (1.5pt);
				\end{scope}
				
			\end{tikzpicture}
		\end{center}
		\caption{{\bf Left:} the ``boundary layer'' of the mesh for a two-dimensional domain $\Omega$. 
			{\bf Right:} an enlarged element from the boundary layer, it doesn't need to be convex. At point $B$, there is only one unknown. The numerical 
			solution $u_h$ on the interface $\overline{AB}$-$\overline{BC}$ is continuous at point $B$, but $u_h$ on the outer-inner side is not continuous at B. Continuity of $u_h$ follows the same pattern at other end points of the inner sides.}
		\label{fig:mesh-element}
	\end{figure}
	
	\noindent Following Guan, et al.\cite{Guan2023}, we require the element $D$ to satisfy the following shape-regularity assumptions:
	\begin{assumption}[Shape Regularity \cite{Guan2023}] \label{assum:shape_reg}
		The element $D$ is shape-regular if it satisfies:
		\begin{itemize}
			\item[\textbf{(A1)}] $D$ is a curvilinear polygon with a $C^{0,1}$ curved side and  diameter $h_D$.
			\item[\textbf{(A2)}] $D$ is star-shaped with respect to a disc $\mathfrak{B}_D \subset D$ with radius $\rho_D h_D$, where $0 < \rho_D < 1/2$;  {here $\rho_D$ is the dimensionless relative radius parameter.}
			\item[\textbf{(A3)}] The parameter $\rho_D$ has a uniform lower bound $0 < \rho_{\min} < \rho_D$.
		\end{itemize}
	\end{assumption}
	\noindent Assumption (A2) is particularly crucial as it mathematically permits the application of the Averaged Taylor Polynomial theory over the curved element $D$. Let $A\lesssim B$ denote $A\le \text{(constant)} B$. 
	
	\subsection{The Weak Galerkin and the Interpolation-Projection Operator}
	Let $V(D)$ be the space of discontinuous functions defined on $D$ and its boundary $\partial D$:
	\begin{equation}
		V(D) = \{ v = (v_0, v_b) : v_0 \in L^2(D), v_b \in L^2(\partial D) \}.
	\end{equation}
	For any $v \in V(D)$, the weak gradient $\nabla_w v \in [\mathbb{P}_{k-1}(D)]^2$ is defined as the unique polynomial satisfying:
	\begin{equation} \label{eq:weak_gradient}
		(\nabla_w v, \vec{q})_D := -(v_0, \nabla \cdot \vec{q})_D + \langle v_b, \vec{q} \cdot \vec{n} \rangle_{\partial D}, \quad \forall \vec{q} \in [\mathbb{P}_{k-1}(D)]^2.
	\end{equation}
	In our scheme, the boundary $\partial D$ is partitioned into two disjoint sets: $\partial D_{proj}$ (the outer and outer-inner sides) and $\partial D_{int}$ (the inner side). We define the interpolation-projection operator $Q_h \xi|_D = (Q_{k,D}^0 \xi, Q_{b,D}^\partial \xi)$ for $\xi \in H^{k+1}(D)$ as follows:
	\begin{definition}[Interpolation-Projection Operator $Q_h$]
	 Let $\partial D_{proj}$ denote the union of the outer-side and outer-inner sides of $D$, and $\partial D_{int}$ denote the inner-side of $D$. We define the projection operator $Q_h \xi|_D = (Q_{k,D}^0 \xi, Q_{b,D}^\partial \xi)$, where $Q_{k,D}^0$ is the $L^2$ projection from $L^2(D)$ to $\mathbb{P}_k(D)$, and on each side $e \subset \partial D$:
	\begin{equation}
	Q_{b,D}^\partial \xi |_e =
	\begin{cases}
		Q_{k,e}^\partial \xi, & \text{if } e \in \partial D_{proj} \text{ ($L^2$ projection to } \mathbb{P}_k|_e\text{)}, \\
		I_{k,e}^\partial \xi, & \text{if } e \in \partial D_{int} \text{ (conforming nodal interpolation to } \mathbb{P}_k(e)\text{)}.
	\end{cases}
	\end{equation}
	\end{definition}
	\noindent
	 {Throughout the paper, boundary traces are understood on the closure $\overline D$. In particular, if
	$w$ is continuous on $\overline D$, then
	$\|w\|_{L^\infty(\overline D)}=\|w\|_{L^\infty(D)}$.
	Thus, for functions continuous on $\overline D$, the norms
	$\|\cdot\|_{L^\infty(D)}$ and $\|\cdot\|_{L^\infty(\overline D)}$ are used interchangeably.}

	\noindent To bound boundary terms by volume terms, we rely on the following standard inequalities established for shape-regular curvilinear elements.
	\begin{lemma}[Trace Inequality {\cite[Lemma 2.2]{Guan2023}}] \label{lem:trace}
		If $D$ is shape-regular, then for any side $e \subset \partial D$ and any $v \in H^1(D)$:
		\begin{equation}
			\|v\|_{L^2(e)}^2 \lesssim h_D^{-1} \|v\|_{L^2(D)}^2 + h_D \|\nabla v\|_{L^2(D)}^2.
		\end{equation}
	\end{lemma}
	
	\begin{lemma}[Discrete Inverse Trace Inequality {\cite[Lemma 2.4]{Guan2023}}] \label{lem:inverse_trace}
		If $D$ is shape-regular, then for any polynomial vector field $\vec{q} \in [\mathbb{P}_{k}(D)]^2$:
		\begin{equation}
			\|\vec{q}\|_{L^2(\partial D)} \lesssim h_D^{-1/2} \|\vec{q}\|_{L^2(D)}.
		\end{equation}
	\end{lemma}
	
	\noindent To properly bound the approximation error over both $L^2$ projections and nodal interpolations simultaneously, we utilize the Averaged Taylor Polynomial (see Definition 4.1.3 in \cite{BrennerScott2008}). 
	
	\begin{lemma}[Averaged Taylor Polynomial Bounds \cite{BrennerScott2008}] \label{lem:taylor_bounds}
		Suppose $D \subset \mathbb{R}^2$ is star-shaped with respect to a disc (satisfying Assumption \ref{assum:shape_reg}). Let $k \ge 1$. For any $\xi \in H^{k+1}(D)$, there exists a polynomial $p \in \mathbb{P}_k(D)$ (the Averaged Taylor Polynomial of degree $k$) such that the following approximation bounds hold simultaneously:
		\begin{align}
			\|\xi - p\|_{L^2(D)} &\lesssim h_D^{k+1} |\xi|_{H^{k+1}(D)}, \label{eq:taylor_L2} \\
			\|\nabla(\xi - p)\|_{L^2(D)} &\lesssim h_D^k |\xi|_{H^{k+1}(D)}, \label{eq:taylor_H1} \\
			\|\xi - p\|_{L^\infty(D)} &\lesssim h_D^k |\xi|_{H^{k+1}(D)}. \label{eq:taylor_Linf}
		\end{align}
	\end{lemma}
	\begin{proof}
		The polynomial $p$ is defined as the Averaged Taylor Polynomial (Definition 4.1.3 in \cite{BrennerScott2008}). The $L^2(D)$ and $H^1(D)$ volume bounds \eqref{eq:taylor_L2} and \eqref{eq:taylor_H1} follow directly from the Bramble-Hilbert Lemma (Lemma 4.3.8 in \cite{BrennerScott2008}) evaluated at $m=k+1$ and $p=2$. The $L^\infty(D)$ bound \eqref{eq:taylor_Linf} follows from Proposition 4.3.2 in \cite{BrennerScott2008} (utilizing the Sobolev embedding $H^{k+1}(D) \hookrightarrow L^\infty(D)$ since $k+1 - 2/2 = k \geq 1$).
	\end{proof}
	
	\noindent The stability property of the 1D nodal interpolation used on the inner side $e \in \partial D_{int}$ is shown in the following Lemma.
	\begin{lemma}[1D Nodal Interpolant Stability] \label{lem:interpolation_stability}
		Let $I_{k,e}^\partial: C^0(e) \to \mathbb{P}_k(e)$ be the standard nodal interpolation of degree $k$ on a 1D edge $e$. For any continuous function $v \in C^0(e)$, we have the $L^\infty$ stability bound:
		\begin{equation}
			\|I_{k,e}^\partial v\|_{L^\infty(e)} \le \Lambda_k \|v\|_{L^\infty(e)},
		\end{equation}
		where $\Lambda_k$ is the Lebesgue constant for the chosen nodal distribution on $\mathbb{P}_k(e)$. Furthermore, $I_{k,e}^\partial p = p$ for any polynomial $p \in \mathbb{P}_k(e)$.
	\end{lemma}	

\begin{lemma}[Estimation of the Interpolation-Projection Operator] \label{lem:I-P-estimation}
	Let $Q_h$ be the operator defined above. Assuming $D$ is shape-regular and $k \ge 1$, for any $\xi \in H^{k+1}(D)$, we have
	\begin{align}
		&(\nabla_w Q_h \xi, \vec{q})_D = (\mathbb{Q}_{k-1,D} \nabla \xi, \vec{q})_D + \langle Q_{b,D}^\partial \xi - \xi, \vec{q} \cdot \vec{n} \rangle_{\partial D}, \quad \forall \vec{q} \in [\mathbb{P}_{k-1}(D)]^2, \label{eq:lem1}\\
		&|\langle Q_{b,D}^\partial \xi - \xi, \vec{q} \cdot \vec{n} \rangle_{\partial D}| \lesssim h_D^k \|\xi\|_{H^{k+1}(D)} \|\vec{q}\|_{L^2(D)}, \label{eq:lem2}\\
		&\|\nabla_w Q_h \xi - \mathbb{Q}_{k-1,D} \nabla \xi\|_{L^2(D)} \lesssim h_D^k \|\xi\|_{H^{k+1}(D)}, \label{eq:lem3}
	\end{align}
	where $\nabla_w$ is the weak gradient operator, $\mathbb{Q}_{k-1,D}$ is the $L^2$ projection onto $[\mathbb{P}_{k-1}(D)]^2$, and the hidden constants depend only on the shape regularity parameters and polynomial degree $k$.
\end{lemma}
\begin{proof}
	\noindent \textbf{Part 1: Prove \eqref{eq:lem1}.} 
	Using the definition of the weak gradient \eqref{eq:weak_gradient}, applying integration by parts, and utilizing the properties of the $L^2$ projections $\mathbb{Q}_{k-1,D}$ and $Q_{k,D}^0$, we have:
	\begin{align*}
		(\nabla_w Q_h \xi, \vec{q})_D &= -(Q_{k,D}^0 \xi, \nabla \cdot \vec{q})_D + \langle Q_{b,D}^\partial \xi, \vec{q} \cdot \vec{n} \rangle_{\partial D} \\
		&= -(\xi, \nabla \cdot \vec{q})_D + \langle \xi, \vec{q} \cdot \vec{n} \rangle_{\partial D} + \langle Q_{b,D}^\partial \xi - \xi, \vec{q} \cdot \vec{n} \rangle_{\partial D} \\
		&= (\nabla \xi, \vec{q})_D + \langle Q_{b,D}^\partial \xi - \xi, \vec{q} \cdot \vec{n} \rangle_{\partial D} \\
		&= (\mathbb{Q}_{k-1,D} \nabla \xi, \vec{q})_D + \langle Q_{b,D}^\partial \xi - \xi, \vec{q} \cdot \vec{n} \rangle_{\partial D}.
	\end{align*}
	
	\noindent \textbf{Part 2: Averaged Taylor Polynomial Estimates.} 
	To prove \eqref{eq:lem2}, we first apply the Cauchy-Schwarz inequality and the discrete inverse trace inequality for polynomials (Lemma \ref{lem:inverse_trace}), which implies $\|\vec{q}\|_{L^2(\partial D)} \lesssim h_D^{-1/2} \|\vec{q}\|_{L^2(D)}$:
	\begin{align} \label{eq:boundary_bound}
		|\langle Q_{b,D}^\partial \xi - \xi, \vec{q} \cdot \vec{n} \rangle_{\partial D}| &\le \sum_{e \subset \partial D} \|Q_{b,D}^\partial \xi - \xi\|_{L^2(e)} \|\vec{q} \cdot \vec{n}\|_{L^2(e)} \nonumber \\
		&\lesssim h_D^{-1/2} \|\vec{q}\|_{L^2(D)} \sum_{e \subset \partial D} \|Q_{b,D}^\partial \xi - \xi\|_{L^2(e)}.
	\end{align}
	We must bound the boundary error $\|Q_{b,D}^\partial \xi - \xi\|_{L^2(e)}$. Let $p \in \mathbb{P}_k(D)$ be the \textbf{Averaged Taylor Polynomial} of degree $k$ of $\xi$ over $D$. Because $D$ satisfies the shape regularity conditions (Assumption \ref{assum:shape_reg}), Lemma \ref{lem:taylor_bounds} guarantees that this single polynomial $p$ satisfies optimal bounds simultaneously in multiple norms. 
	
	\noindent Specifically, Lemma \ref{lem:taylor_bounds} provides the $L^\infty(D)$ volume bound:
	\begin{equation}
		\|\xi - p\|_{L^\infty(D)} \lesssim h_D^k |\xi|_{H^{k+1}(D)} \le h_D^k \|\xi\|_{H^{k+1}(D)}, \label{eq:taylor2}
	\end{equation}
	as well as the standard $L^2(D)$ and $H^1(D)$ volume estimates:
	\begin{equation}
		\|\xi - p\|_{L^2(D)} \lesssim h_D^{k+1} \|\xi\|_{H^{k+1}(D)} \quad \text{and} \quad \|\nabla(\xi - p)\|_{L^2(D)} \lesssim h_D^k \|\xi\|_{H^{k+1}(D)}.
	\end{equation}
	Substituting these into the continuous trace inequality (Lemma \ref{lem:trace}) applied to the error $v = \xi - p$, we obtain the $L^2(e)$ boundary trace bound:
	\begin{align}
		\|\xi - p\|_{L^2(e)}^2 &\lesssim h_D^{-1} \|\xi - p\|_{L^2(D)}^2 + h_D \|\nabla(\xi - p)\|_{L^2(D)}^2 \nonumber \\
		&\lesssim h_D^{-1} \left(h_D^{k+1} \|\xi\|_{H^{k+1}(D)}\right)^2 + h_D \left(h_D^k \|\xi\|_{H^{k+1}(D)}\right)^2 \nonumber \\
		&\lesssim h_D^{2k+1} \|\xi\|_{H^{k+1}(D)}^2,
	\end{align}
	which yields:
	\begin{equation}
		\|\xi - p\|_{L^2(e)} \lesssim h_D^{k+1/2} \|\xi\|_{H^{k+1}(D)}. \label{eq:taylor1}
	\end{equation}
	
	\noindent \textbf{Part 3: Boundary Error Bounds.} 
	We split the sum over the boundary edges into two cases, utilizing the same polynomial $p$ for both to apply the triangle inequality.
	\medskip
	
	\noindent \textit{Case 1: Outer and outer-inner sides ($e \in \partial D_{proj}$).} \\
	Here, $Q_{b,D}^\partial$ is the $L^2(e)$ projection, which preserves polynomials of degree $k$ (i.e., $Q_{k,e}^\partial p = p$). By the $L^2$ projection's best-approximation property and \eqref{eq:taylor1}:
	\begin{equation}
		\|Q_{k,e}^\partial \xi - \xi\|_{L^2(e)} \le \|\xi - p\|_{L^2(e)} \lesssim h_D^{k+1/2} \|\xi\|_{H^{k+1}(D)}.
	\end{equation}
	
	\noindent \textit{Case 2: Inner side ($e \in \partial D_{int}$).} \\
	Here, $Q_{b,D}^\partial$ is the conforming nodal interpolation operator $I_{k,e}^\partial$, which also preserves polynomials of degree $k$ (i.e., $I_{k,e}^\partial p = p$ by Lemma \ref{lem:interpolation_stability}). By the triangle inequality:
	\begin{equation}\label{interp-on-e-1}
		\|I_{k,e}^\partial \xi - \xi\|_{L^2(e)} \le \|\xi - p\|_{L^2(e)} + \|I_{k,e}^\partial (\xi - p)\|_{L^2(e)}.
	\end{equation}
	The first term is bounded by $h_D^{k+1/2} \|\xi\|_{H^{k+1}(D)}$ via \eqref{eq:taylor1}. For the second term, we use the $L^\infty(e)$ stability of the 1D nodal interpolation (Lemma \ref{lem:interpolation_stability}), followed by the volume bound \eqref{eq:taylor2}:
	\begin{align}\label{interp-on-e-2}
		\|I_{k,e}^\partial (\xi - p)\|_{L^2(e)} &\lesssim h_e^{1/2} \|I_{k,e}^\partial (\xi - p)\|_{L^\infty(e)} \nonumber \\
		&\lesssim h_D^{1/2} \|\xi - p\|_{L^\infty(e)} \nonumber \\
		&\le  {h_D^{1/2} \|\xi - p\|_{L^\infty(D)}} \nonumber \\
		&\lesssim h_D^{1/2} \left( h_D^k \|\xi\|_{H^{k+1}(D)} \right) = h_D^{k+1/2} \|\xi\|_{H^{k+1}(D)}.
	\end{align}
	Summing these bounds gives $\|I_{k,e}^\partial \xi - \xi\|_{L^2(e)} \lesssim h_D^{k+1/2} \|\xi\|_{H^{k+1}(D)}$.
	
	\noindent \textbf{Part 4: Conclusion of the Proof.} 
	Combining the estimates from Case 1 and Case 2 back into \eqref{eq:boundary_bound}, we obtain:
	\begin{equation}
		|\langle Q_{b,D}^\partial \xi - \xi, \vec{q} \cdot \vec{n} \rangle_{\partial D}| \lesssim h_D^{-1/2} \|\vec{q}\|_{L^2(D)} \left( h_D^{k+1/2} \|\xi\|_{H^{k+1}(D)} \right) = h_D^k \|\xi\|_{H^{k+1}(D)} \|\vec{q}\|_{L^2(D)},
	\end{equation}
	which proves \eqref{eq:lem2}.
	
	\noindent Finally, to prove \eqref{eq:lem3}, let $\vec{q} = \nabla_w Q_h \xi - \mathbb{Q}_{k-1,D} \nabla \xi \in [\mathbb{P}_{k-1}(D)]^2$. Rearranging \eqref{eq:lem1}, we get:
	\begin{equation}
		\|\vec{q}\|_{L^2(D)}^2 = (\vec{q}, \vec{q})_D = \langle Q_{b,D}^\partial \xi - \xi, \vec{q} \cdot \vec{n} \rangle_{\partial D}.
	\end{equation}
	Applying the newly established bound \eqref{eq:lem2}, we have:
	\begin{equation}
		\|\vec{q}\|_{L^2(D)}^2 \lesssim h_D^k \|\xi\|_{H^{k+1}(D)} \|\vec{q}\|_{L^2(D)}.
	\end{equation}
	 {If $\|\vec{q}\|_{L^2(D)}=0$, then \eqref{eq:lem3} is immediate. Otherwise,} dividing both sides by
	$\|\vec{q}\|_{L^2(D)}$ yields \eqref{eq:lem3} and completes the proof.
\end{proof}

\section{The Coupled WG-CG Method}
\label{sec:WG-CG scheme}

\subsection{Domain Partition and Function Spaces}
Let $\Omega \subset \mathbb{R}^2$ be a bounded domain with a curved boundary $\partial \Omega$. The domain is partitioned into two disjoint subdomains: the boundary layer $\Omega_{wg}$ (a union of curvilinear elements $D$, see Figure \ref{fig:mesh-element}) and the interior polygonal region $\Omega_{cg}$ (a union of triangular elements $T$). We denote the mesh partitions as $\mathcal{T}_h^{wg}$ and $\mathcal{T}_h^{cg}$, respectively. Let $h:=\max_{K\in\mathcal T_h} h_K$, where $\mathcal{T}_h = \mathcal{T}_h^{wg}\cup \mathcal{T}_h^{cg}$, $K$ can be $D$ or $T$.  The interface between the two regions is the boundary of a polygon denoted by $\Gamma = \partial \Omega_{wg} \cap \partial \Omega_{cg}$.

\noindent For an integer $k \ge 1$, we define the continuous finite element space on the interior as:
\begin{equation}
	V_h^{cg} = \{ v_c \in C^0(\Omega_{cg}) : v_c|_T \in \mathbb{P}_k(T), \forall T \in \mathcal{T}_h^{cg} \}.
\end{equation}
On the boundary layer, we define the weak Galerkin space as:
\begin{equation}
	V_h^{wg} = \{ v_{w} = (v_0, v_b) \text{ on } \Omega_{wg} : v_0|_D \in \mathbb{P}_k(D), v_b|_e \in \mathbb{P}_k|_{e}, e \subset \partial D, \forall D \in \mathcal{T}_h^{wg} \},
\end{equation}
and $v_b$ is single-valued on WG edges. The global hybrid function space enforces strong continuity across the interface $\Gamma$ and the zero Dirichlet boundary condition on $\partial \Omega$:
\begin{equation}
	V_h = \left\{ v = (v_c, v_w) : v_c \in V_h^{cg}, v_w \in V_h^{wg}, \ v_w|_\Gamma =v_b|_\Gamma = v_c|_\Gamma, \ v_w|_{\partial \Omega} = 0 \right\}.
\end{equation}
For any $v \in V_h$, the weak gradient $\nabla_w v \in [\mathbb{P}_{k-1}(D)]^2$ on $D \in \mathcal{T}_h^{wg}$ is defined by:
\begin{equation}
	(\nabla_w v, \vec{q})_D = -(v_0, \nabla \cdot \vec{q})_D + \langle v_b, \vec{q} \cdot \mathbf{n}_{wg} \rangle_{\partial D}, \quad \forall \vec{q} \in [\mathbb{P}_{k-1}(D)]^2.
\end{equation}

\subsection{The Numerical Scheme}
We consider the Poisson equation $$-\Delta u = f$$ in $\Omega$ with $u = 0$ on $\partial \Omega$. The coupled WG-CG scheme seeks $u_h = (u_c, u_w) \in V_h$ such that:
\begin{equation} \label{eq:scheme}
	A_h(u_h, v) = (f, v_0)_{\Omega_{wg}} + (f, v_c)_{\Omega_{cg}}, \quad \forall v \in V_h,
\end{equation}
where the global bilinear form is defined as:
\begin{equation}
	A_h(u, v) = \sum_{D \in \mathcal{T}_h^{wg}} \Bigl( (\nabla_w u_w, \nabla_w v_w)_D + s_D(u_w, v_w) \Bigl) + \sum_{T \in \mathcal{T}_h^{cg}} (\nabla u_c, \nabla v_c)_T,
\end{equation}
and the stabilization term on the WG elements is
$$s_D(u_w, v_w) = h_D^{-1} \langle u_0 - u_b, v_0 - v_b \rangle_{\partial D}.$$

\noindent We define the energy norm $||| \cdot |||$ on $V_h$ as:
\begin{equation}
	||| v |||^2 = \sum_{D \in \mathcal{T}_h^{wg}} \left( \|\nabla_w v_w\|_{L^2(D)}^2 + s_D(v_w, v_w) \right) + \sum_{T \in \mathcal{T}_h^{cg}} \|\nabla v_c\|_{L^2(T)}^2.
\end{equation}

\begin{theorem}[Unique Solvability] \label{thm:unique_solvability}
	The coupled WG-CG numerical scheme \eqref{eq:scheme} has a unique solution $u_h \in V_h$.
\end{theorem}

\begin{proof}
	Since the numerical scheme \eqref{eq:scheme} reduces to a square finite-dimensional linear system, it suffices to prove uniqueness. Let $v = (v_c, v_w) \in V_h$, with $v_w = (v_0, v_b)$, be the solution to the corresponding homogeneous problem (i.e., with $f = 0$):
	\begin{equation}
		A_h(v, v) = 0.
	\end{equation}
	By the definition of the bilinear form, this implies $||| v |||^2 = A_h(v,v) = 0$. Consequently, we obtain the following three conditions:
	\begin{enumerate}
		\item $\|\nabla v_c\|_{L^2(T)} = 0$ for all $T \in \mathcal{T}_h^{cg}$,
		\item $\|v_0 - v_b\|_{L^2(\partial D)} = 0 \implies v_0 = v_b$ on $\partial D$ for all $D \in \mathcal{T}_h^{wg}$,
		\item $\|\nabla_w v_w\|_{L^2(D)} = 0 \implies \nabla_w v_w = \mathbf{0}$ for all $D \in \mathcal{T}_h^{wg}$.
	\end{enumerate}
	From condition 1, since $v_c \in C^0(\Omega_{cg})$, $v_c$ must be a global constant everywhere in $\Omega_{cg}$.
	
	\noindent On each WG element $D$, because $v_0 \in \mathbb{P}_k(D)$, we have $\nabla v_0 \in[\mathbb{P}_{k-1}(D)]^2$. Therefore, we can choose $\vec{q} = \nabla v_0$ as the test function in the definition of the weak gradient to obtain:
	\begin{equation}
		0 = (\nabla_w v_w, \nabla v_0)_D = -(v_0, \Delta v_0)_D + \langle v_b, \nabla v_0 \cdot \mathbf{n}_{wg} \rangle_{\partial D}.
	\end{equation}
	Applying standard integration by parts to the first term yields:
	\begin{equation}
		-(v_0, \Delta v_0)_D = (\nabla v_0, \nabla v_0)_D - \langle v_0, \nabla v_0 \cdot \mathbf{n}_{wg} \rangle_{\partial D}.
	\end{equation}
	Substituting this back, we get:
	\begin{equation} \label{eq:solvability_step}
		0 = \|\nabla v_0\|_{L^2(D)}^2 - \langle v_0 - v_b, \nabla v_0 \cdot \mathbf{n}_{wg} \rangle_{\partial D}.
	\end{equation}
	By condition 2 ($v_0 = v_b$ on $\partial D$), the boundary integral in \eqref{eq:solvability_step} vanishes. Thus, $\|\nabla v_0\|_{L^2(D)} = 0$, which implies $v_0$ is a constant on each element $D$. Since $v_0 = v_b$ on $\partial D$, the boundary trace $v_b$ shares this exact same constant value. Because $v \in V_h$, the function satisfies the homogeneous Dirichlet boundary condition, meaning $v_b = 0$ on $\partial \Omega$. For any WG element $D$ that shares an edge with $\partial \Omega$, its internal constant value must therefore be zero. By propagating this zero value across adjacent WG elements via the single-valued interfaces $v_b$, we find that $v_w = (v_0, v_b) = (0, 0)$ globally throughout the boundary layer $\Omega_{wg}$.  {Finally, the space $V_h$ enforces continuity across the interface $\Gamma$, yielding $v_c|_\Gamma = v_w|_\Gamma = 0$. Since $v_c$ is a global constant on $\Omega_{cg}$, we conclude that $v_c = 0$ everywhere in $\Omega_{cg}$. Therefore, $v = 0$ over the entire domain $\Omega$, proving that the homogeneous problem has only the trivial solution. This guarantees the unique solvability of the scheme.}
\end{proof}

\section{Error Analysis}
\label{sec:error estimates}

\subsection{Global Interpolation-Projection Operator and the Error Equation}
Let $u$ be the exact solution. We define the global interpolation-projection  {operator} $\Pi_h u \in V_h$ as follows:
\begin{itemize}
	\item On $\Omega_{cg}$, $\Pi_h u = I_h u$, the standard continuous nodal interpolation of degree $k$.
	\item On $\Omega_{wg}$, $\Pi_h u|_{D} = Q_h u|_{D} = (Q_{k,D}^0 u, Q_{b,D}^\partial u)$, where $Q_{b,D}^\partial u$ uses the $L^2$ projection on edges $e \subset  \partial D_{proj}$, and the conforming nodal interpolation $I_{k,e}^\partial u$ on the interface $e \subset \Gamma$.
\end{itemize}
Because $I_{k,e}^\partial u$ perfectly matches the trace of the interior nodal interpolation $I_h u$ on $\Gamma$, we have 
$$Q_h u|_\Gamma = I_h u|_\Gamma,$$ 
ensuring $\Pi_h u \in V_h$.

\begin{lemma}[The Error Equation] \label{lem:error_eq}
	Let $u \in H^{k+1}(\Omega)\cap H_0^1(\Omega)$ be the exact solution of $-\Delta u=f$ in $\Omega$, and let
	$u_h=(u_c,u_w)\in V_h$ solve \eqref{eq:scheme}. Define the global interpolation--projection
	$\Pi_h u\in V_h$ by $\Pi_h u|_{\Omega_{cg}}=I_h u$ and $\Pi_h u|_{D}=Q_h u$ on each
	$D\in\mathcal T_h^{wg}$, and set $e_h:=\Pi_h u-u_h\in V_h$.
	Then for any $v=(v_c,v_w)\in V_h$ (with $v_w=(v_0,v_b)$), we have
	\begin{align}\label{eq:error_eq_1}
		A_h(e_h,v)
		&=\sum_{D\in\mathcal T_h^{wg}}
		\Big\langle (\nabla u-\mathbb Q_{k-1,D}\nabla u)\cdot \mathbf n_{wg},\, v_0-v_b\Big\rangle_{\partial D}
		+\sum_{D\in\mathcal T_h^{wg}}
		\Big\langle Q_{b,D}^\partial u-u,\, \nabla_w v_w\cdot \mathbf n_{wg}\Big\rangle_{\partial D}\nonumber\\
		&\quad +\sum_{D\in\mathcal T_h^{wg}} s_D(Q_h u, v_w)
		+\sum_{T\in\mathcal T_h^{cg}} (\nabla I_h u-\nabla u,\nabla v_c)_T .
	\end{align}
\end{lemma}

\begin{proof}
	Fix $v=(v_c,v_w)\in V_h$ and write $v_w=(v_0,v_b)$ on each $D\in\mathcal T_h^{wg}$.
	Since $\Pi_h u=(I_h u, Q_h u)$, we begin by expanding
	\[
	A_h(\Pi_h u, v)
	=\sum_{D\in\mathcal T_h^{wg}}\Big( (\nabla_w Q_h u,\nabla_w v_w)_D+s_D(Q_h u,v_w)\Big)
	+\sum_{T\in\mathcal T_h^{cg}}(\nabla I_h u,\nabla v_c)_T .
	\]
	
	\medskip\noindent
	\textbf{Step 1: A local identity on each $D\in\mathcal T_h^{wg}$.}
	Take $\vec q=\nabla_w v_w\in[\mathbb P_{k-1}(D)]^2$ in Lemma~\ref{lem:I-P-estimation} (identity \eqref{eq:lem1})
	with $\xi=u$, to obtain
	\begin{equation}\label{eq:step1}
		(\nabla_w Q_h u,\nabla_w v_w)_D
		=(\mathbb Q_{k-1,D}\nabla u,\nabla_w v_w)_D
		+\big\langle Q_{b,D}^\partial u-u,\, \nabla_w v_w\cdot \mathbf n_{wg}\big\rangle_{\partial D}.
	\end{equation}
	By the definition of $\nabla_w$ with test vector $\mathbb Q_{k-1,D}\nabla u\in[\mathbb P_{k-1}(D)]^2$, we have
	\begin{align}\label{eq:step1-t1}
	(\mathbb Q_{k-1,D}\nabla u, \nabla_w v_w)_D
	&=-(\nabla\!\cdot(\mathbb Q_{k-1,D}\nabla u),v_0)_D
	+\big\langle (\mathbb Q_{k-1,D}\nabla u)\cdot \mathbf n_{wg}, v_b\big\rangle_{\partial D} \nonumber \\
	&=(\mathbb Q_{k-1,D}\nabla u,\nabla v_0)_D
	-\big\langle (\mathbb Q_{k-1,D}\nabla u)\cdot \mathbf n_{wg},\, v_0-v_b\big\rangle_{\partial D} \nonumber \\
	&=(\nabla u,\nabla v_0)_D
	-\big\langle (\mathbb Q_{k-1,D}\nabla u)\cdot \mathbf n_{wg},\, v_0-v_b\big\rangle_{\partial D},
	\end{align}
	where $\nabla v_0\in[\mathbb P_{k-1}(D)]^2$. Then, using $-\Delta u=f$ on $D$ and integrating by parts,
	\[
	(\nabla u,\nabla v_0)_D = (f,v_0)_D + \big\langle \nabla u\cdot \mathbf n_{wg},\, v_0\big\rangle_{\partial D}.
	\]
	Replacing $(\nabla u,\nabla v_0)_D$ by its right hand side, \eqref{eq:step1-t1} gives
	\begin{align}\label{eq:step1b}
		(\mathbb Q_{k-1,D}\nabla u,\nabla_w v_w)_D
		&=(f,v_0)_D
		+\big\langle \nabla u\cdot \mathbf n_{wg},\, v_0\big\rangle_{\partial D}
		-\big\langle (\mathbb Q_{k-1,D}\nabla u)\cdot \mathbf n_{wg},\, v_0-v_b\big\rangle_{\partial D}\nonumber\\
		&=(f,v_0)_D
		+\big\langle \nabla u\cdot \mathbf n_{wg},\, v_b\big\rangle_{\partial D}
		+\big\langle (\nabla u-\mathbb Q_{k-1,D}\nabla u)\cdot \mathbf n_{wg},\, v_0-v_b\big\rangle_{\partial D}.
	\end{align}
	Substituting \eqref{eq:step1b} into \eqref{eq:step1} yields the local representation
	\begin{align}\label{eq:local_wg_rep}
		(\nabla_w Q_h u,\nabla_w v_w)_D
		&=(f,v_0)_D
		+\big\langle \nabla u\cdot \mathbf n_{wg},\, v_b\big\rangle_{\partial D}
		+\big\langle (\nabla u-\mathbb Q_{k-1,D}\nabla u)\cdot \mathbf n_{wg},\, v_0-v_b\big\rangle_{\partial D}\nonumber\\
		&\quad+\big\langle Q_{b,D}^\partial u-u,\, \nabla_w v_w\cdot \mathbf n_{wg}\big\rangle_{\partial D}.
	\end{align}
	
	\medskip\noindent
	\textbf{Step 2: Summation over $\mathcal T_h^{wg}$ and boundary/interface reduction.}
	Summing \eqref{eq:local_wg_rep} over all $D\in\mathcal T_h^{wg}$ gives
	\begin{align}\label{eq:wg_sum}
		\sum_{D\in\mathcal T_h^{wg}}(\nabla_w Q_h u,\nabla_w v_w)_D
		&=(f,v_0)_{\Omega_{wg}}
		+\sum_{D\in\mathcal T_h^{wg}}\big\langle \nabla u\cdot \mathbf n_{wg},\, v_b\big\rangle_{\partial D}\nonumber\\
		&\quad+\sum_{D\in\mathcal T_h^{wg}}
		\big\langle (\nabla u-\mathbb Q_{k-1,D}\nabla u)\cdot \mathbf n_{wg},\, v_0-v_b\big\rangle_{\partial D}\nonumber\\
		&\quad+\sum_{D\in\mathcal T_h^{wg}}
		\big\langle Q_{b,D}^\partial u-u,\, \nabla_w v_w\cdot \mathbf n_{wg}\big\rangle_{\partial D}.
	\end{align}
	The flux term $\displaystyle \sum_D \langle \nabla u\cdot \mathbf n_{wg}, v_b\rangle_{\partial D}$ cancels on outer-inner sides
	of $D$ in $\Omega_{wg}$ (opposite normals and single-valued traces), and $v_b=0$ on $\partial\Omega$ by the definition
	of $V_h$. Hence,
	\begin{equation}\label{eq:flux_reduce}
		\sum_{D\in\mathcal T_h^{wg}}\big\langle \nabla u\cdot \mathbf n_{wg},\, v_b\big\rangle_{\partial D}
		=\big\langle \nabla u\cdot \mathbf n_{wg},\, v_b\big\rangle_{\Gamma}.
	\end{equation}
	
	\medskip\noindent
	\textbf{Step 3: The CG part and interface cancellation.}
	On $\Omega_{cg}$,
	\[
	\sum_{T\in\mathcal T_h^{cg}}(\nabla I_h u,\nabla v_c)_T
	=\sum_{T\in\mathcal T_h^{cg}}(\nabla u,\nabla v_c)_T
	+\sum_{T\in\mathcal T_h^{cg}}(\nabla I_h u-\nabla u,\nabla v_c)_T.
	\]
	Using $-\Delta u=f$ and integrating by parts over $\Omega_{cg}$ (whose boundary is $\Gamma$),
	\[
	(\nabla u,\nabla v_c)_{\Omega_{cg}}=(f,v_c)_{\Omega_{cg}}+\big\langle \nabla u\cdot \mathbf n_{cg},\, v_c\big\rangle_{\Gamma}.
	\]
	Therefore,
	\begin{equation}\label{eq:cg_rep}
		\sum_{T\in\mathcal T_h^{cg}}(\nabla I_h u,\nabla v_c)_T
		=(f,v_c)_{\Omega_{cg}}+\big\langle \nabla u\cdot \mathbf n_{cg},\, v_c\big\rangle_{\Gamma}
		+\sum_{T\in\mathcal T_h^{cg}}(\nabla I_h u-\nabla u,\nabla v_c)_T.
	\end{equation}
	On the interface $\Gamma$, the outward normals satisfy $\mathbf n_{wg}=-\mathbf n_{cg}$ and
	$v\in V_h$ enforces $v_b=v_c$ on $\Gamma$. Hence,
	\[
	\big\langle \nabla u\cdot \mathbf n_{wg},\, v_b\big\rangle_{\Gamma}
	+\big\langle \nabla u\cdot \mathbf n_{cg},\, v_c\big\rangle_{\Gamma}=0.
	\]
	
	\medskip\noindent
	\textbf{Step 4: Assemble $A_h(\Pi_h u,v)$ and subtract the discrete scheme.}
	Insert \eqref{eq:wg_sum}--\eqref{eq:flux_reduce} and \eqref{eq:cg_rep} into the expansion of $A_h(\Pi_h u,v)$.
	After the interface cancellation, we obtain
	\begin{align*}
		A_h(\Pi_h u,v)
		&=(f,v_0)_{\Omega_{wg}}+(f,v_c)_{\Omega_{cg}}
		+\sum_{D\in\mathcal T_h^{wg}}
		\big\langle (\nabla u-\mathbb Q_{k-1,D}\nabla u)\cdot \mathbf n_{wg},\, v_0-v_b\big\rangle_{\partial D}\\
		&\quad+\sum_{D\in\mathcal T_h^{wg}}
		\big\langle Q_{b,D}^\partial u-u,\, \nabla_w v_w\cdot \mathbf n_{wg}\big\rangle_{\partial D}
		+\sum_{D\in\mathcal T_h^{wg}} s_D(Q_h u, v_w)
		+\sum_{T\in\mathcal T_h^{cg}} (\nabla I_h u-\nabla u,\nabla v_c)_T.
	\end{align*}
	Finally, subtract the discrete equation \eqref{eq:scheme},
	\(
	A_h(u_h,v)=(f,v_0)_{\Omega_{wg}}+(f,v_c)_{\Omega_{cg}},
	\)
	to get \eqref{eq:error_eq_1} with $e_h=\Pi_h u-u_h$.
\end{proof}

\subsection{Energy Norm Estimate}

We first record two local approximation facts that hold on each shape-regular curvilinear element
$D\in\mathcal T_h^{wg}$ (star-shaped with uniform chunkiness), and on each $T\in\mathcal T_h^{cg}$.
They follow from the averaged Taylor polynomial (Bramble--Hilbert) argument on $D$ and standard
Lagrange interpolation theory on $T$.

\begin{lemma}[Local approximation estimates]\label{lem:local_approx}
	Let $k\ge 1$ and $u\in H^{k+1}(D)$.
	\begin{align}
		&\|u-Q_{k,D}^0u\|_{L^2(D)}+h_D\|\nabla(u-Q_{k,D}^0u)\|_{L^2(D)}
		\lesssim h_D^{k+1}\|u\|_{H^{k+1}(D)}, \label{eq:approx_Q0}\\
		&\|\nabla u-\mathbb Q_{k-1,D}\nabla u\|_{L^2(D)}+h_D\|\nabla(\nabla u-\mathbb Q_{k-1,D}\nabla u)\|_{L^2(D)}
		\lesssim h_D^{k}\|u\|_{H^{k+1}(D)}. \label{eq:approx_Qgrad}
	\end{align}
	Moreover, for any side $e\subset\partial D$,
	\begin{align}
		&\|u-Q_{k,D}^0u\|_{L^2(e)} \lesssim h_D^{k+1/2}\|u\|_{H^{k+1}(D)}, \label{eq:trace_Q0}\\
		&\|(\nabla u-\mathbb Q_{k-1,D}\nabla u)\cdot \mathbf n_{wg}\|_{L^2(e)}
		\lesssim h_D^{k-1/2}\|u\|_{H^{k+1}(D)}. \label{eq:trace_Qgrad}
	\end{align}
	Finally, on each $T\in\mathcal T_h^{cg}$,
	\begin{equation}\label{eq:approx_Ih}
		\|\nabla(u-I_hu)\|_{L^2(T)}\lesssim h_T^{k}\|u\|_{H^{k+1}(T)}.
	\end{equation}
\end{lemma}

\begin{proof}
	The volume estimates \eqref{eq:approx_Q0}--\eqref{eq:approx_Qgrad} follow by taking the averaged Taylor
	polynomial $p\in\mathbb P_k(D)$ (respectively $\vec p\in[\mathbb P_{k-1}(D)]^2$ for $\nabla u$) as an
	approximant and then using the best-approximation property of $L^2$-projections together with the
	Bramble--Hilbert bounds on star-shaped domains. The boundary estimates
	\eqref{eq:trace_Q0}--\eqref{eq:trace_Qgrad} follow from the trace inequality on $D$ applied to
	$u-Q_{k,D}^0u$ and to $(\nabla u-\mathbb Q_{k-1,D}\nabla u)\cdot \mathbf n_{wg}$, using the corresponding
	volume bounds. The interpolation estimate \eqref{eq:approx_Ih} is standard for $C^0$ Lagrange elements
	on triangles.
\end{proof}

\begin{lemma}[Stabilization consistency on $D$]\label{lem:stab_consistency}
	Let $Q_hu|_D=(Q_{k,D}^0u,Q_{b,D}^\partial u)$ be the mixed interpolation--projection operator.
	Then for $u\in H^{k+1}(D)$,
	\begin{equation}\label{eq:stab_consistency}
		s_D(Q_hu,Q_hu)^{1/2}=h_D^{-1/2}\|Q_{k,D}^0u-Q_{b,D}^\partial u\|_{L^2(\partial D)}
		\lesssim h_D^{k}\|u\|_{H^{k+1}(D)}.
	\end{equation}
\end{lemma}

\begin{proof}
	By the triangle inequality,
	\[
	\|Q_{k,D}^0u-Q_{b,D}^\partial u\|_{L^2(\partial D)}
	\le \|Q_{k,D}^0u-u\|_{L^2(\partial D)}+\|u-Q_{b,D}^\partial u\|_{L^2(\partial D)}.
	\]
	The first term is bounded by \eqref{eq:trace_Q0} from Lemma~\ref{lem:local_approx}.
	
	\noindent For the second term, on each edge $e\subset\partial D$ we have either:
	(i) $Q_{b,D}^\partial=Q_{k,e}^\partial$ ($L^2$-projection), so
	$$\|u-Q_{b,D}^\partial u\|_{L^2(e)}\le \|u-p\|_{L^2(e)}\lesssim h_D^{k+1/2}\|u\|_{H^{k+1}(D)}$$
	using the same averaged Taylor polynomial $p\in\mathbb P_k(D)$ and the trace estimate; or
	(ii) $Q_{b,D}^\partial=I_{k,e}^\partial$ (nodal interpolant), and the bound
	$$\|u-I_{k,e}^\partial u\|_{L^2(e)}\lesssim h_D^{k+1/2}\|u\|_{H^{k+1}(D)}$$ follows from the same argument
	as in Lemma~\ref{lem:I-P-estimation} from \eqref{interp-on-e-1} to \eqref{interp-on-e-2}, using the stability of $I_{k,e}^\partial$ and the
	 {$L^\infty$ bound for $u-p$ provided by Lemma~\ref{lem:taylor_bounds}, which follows from the averaged Taylor polynomial estimate together with the Sobolev embedding $H^{k+1}(D)\hookrightarrow L^\infty(D)$ for $k\ge 1$ in two dimensions.}
	Summing over all sides of $\partial D$ yields
	\begin{equation}\label{err:u-QbDpartial-u}
	\|u-Q_{b,D}^\partial u\|_{L^2(\partial D)}\lesssim h_D^{k+1/2}\|u\|_{H^{k+1}(D)}.
	\end{equation}
	Multiplying by $h_D^{-1/2}$ proves \eqref{eq:stab_consistency}.
\end{proof}

\begin{theorem}[Energy norm estimate]\label{thm:energy_norm}
	Let $k\ge 1$ and assume $u\in H^{k+1}(\Omega)\cap H_0^1(\Omega)$ solves $-\Delta u=f$ in $\Omega$.
	Let $u_h\in V_h$ solve the coupled WG--CG scheme \eqref{eq:scheme}, and let $\Pi_hu\in V_h$ be the
	global interpolation--projection operator.
	Then there exists a constant $C>0$, independent of $h$, such that
	\begin{equation}\label{eq:energy_est}
		|||\,\Pi_hu-u_h\,||| \;\le\; C\, h^{k}\,\|u\|_{H^{k+1}(\Omega)}.
	\end{equation}
\end{theorem}

\begin{proof}
	Set $e_h:=\Pi_hu-u_h\in V_h$. Taking $v=e_h$ in the error equation
	\eqref{eq:error_eq_1} gives
	\begin{align*}
		|||e_h|||^2
		&=\sum_{D\in\mathcal T_h^{wg}}
		\Big\langle (\nabla u-\mathbb Q_{k-1,D}\nabla u)\cdot \mathbf n_{wg},\, e_0-e_b\Big\rangle_{\partial D}
		+\sum_{D\in\mathcal T_h^{wg}}
		\Big\langle Q_{b,D}^\partial u-u,\, \nabla_w e_w\cdot \mathbf n_{wg}\Big\rangle_{\partial D}\\
		&\quad+\sum_{D\in\mathcal T_h^{wg}} s_D(Q_hu,e_w)
		+\sum_{T\in\mathcal T_h^{cg}}(\nabla I_hu-\nabla u,\nabla e_c)_T\\
		&=: J_1+J_2+J_3+J_4.
	\end{align*}
	We estimate each term by $C h^k\|u\|_{H^{k+1}(\Omega)}|||e_h|||$.
	
	\medskip\noindent
	\textbf{Estimate of $J_1$.}
	By Cauchy--Schwarz and \eqref{eq:trace_Qgrad} from Lemma~\ref{lem:local_approx},
	\[
	|J_1|
	\le \sum_D \|(\nabla u-\mathbb Q_{k-1,D}\nabla u)\cdot \mathbf n_{wg}\|_{L^2(\partial D)}
	\|e_0-e_b\|_{L^2(\partial D)}
	\lesssim \sum_D h_D^{k-1/2}\|u\|_{H^{k+1}(D)} \|e_0-e_b\|_{L^2(\partial D)}.
	\]
	Using $s_D(e_w,e_w)=h_D^{-1}\|e_0-e_b\|_{L^2(\partial D)}^2$, we have
	$\|e_0-e_b\|_{L^2(\partial D)} = h_D^{1/2}s_D(e_w,e_w)^{1/2}$, hence
	\[
	|J_1| \lesssim \sum_D h_D^{k}\|u\|_{H^{k+1}(D)}\, s_D(e_w,e_w)^{1/2}
	\le \Big(\sum_D h_D^{2k}\|u\|_{H^{k+1}(D)}^2\Big)^{1/2}
	\Big(\sum_D s_D(e_w,e_w)\Big)^{1/2}.
	\]
	Then $\sum_D h_D^{2k}\|u\|_{H^{k+1}(D)}^2 \le h^{2k}\|u\|_{H^{k+1}(\Omega_{wg})}^2$,
	and $\big(\sum_D s_D(e_w,e_w)\big)^{1/2}\le |||e_h|||$. Therefore
	\begin{equation}\label{eq:J1_bound}
		|J_1|\lesssim h^k\|u\|_{H^{k+1}(\Omega_{wg})}\,|||e_h||| \;\le\; h^k\|u\|_{H^{k+1}(\Omega)}\,|||e_h|||.
	\end{equation}
	
	\medskip\noindent
	\textbf{Estimate of $J_2$.}
	Apply Lemma~\ref{lem:I-P-estimation} \eqref{eq:lem2} with $\xi=u$ and
	$\vec q=\nabla_we_w\in[\mathbb P_{k-1}(D)]^2$ to obtain, on each $D$,
	\[
	\Big|\big\langle Q_{b,D}^\partial u-u,\,\nabla_we_w\cdot \mathbf n_{wg}\big\rangle_{\partial D}\Big|
	\lesssim h_D^k\|u\|_{H^{k+1}(D)}\|\nabla_we_w\|_{L^2(D)}.
	\]
	Summing over $D$ and using Cauchy--Schwarz gives
	\begin{equation}\label{eq:J2_bound}
		|J_2|\lesssim h^k\|u\|_{H^{k+1}(\Omega_{wg})}
		\Big(\sum_D \|\nabla_we_w\|_{L^2(D)}^2\Big)^{1/2}
		\le h^k\|u\|_{H^{k+1}(\Omega_{wg})}\,|||e_h||| .
	\end{equation}
	
	\medskip\noindent
	\textbf{Estimate of $J_3$.}
	By Cauchy--Schwarz in the stabilization inner product,
	\[
	|J_3|=\Big|\sum_D s_D(Q_hu,e_w)\Big|
	\le \Big(\sum_D s_D(Q_hu,Q_hu)\Big)^{1/2}
	\Big(\sum_D s_D(e_w,e_w)\Big)^{1/2}.
	\]
	Using Lemma~\ref{lem:stab_consistency} and $h_D\le h$,
	\[
	\Big(\sum_D s_D(Q_hu,Q_hu)\Big)^{1/2}
	\lesssim \Big(\sum_D h_D^{2k}\|u\|_{H^{k+1}(D)}^2\Big)^{1/2}
	\le h^k\|u\|_{H^{k+1}(\Omega_{wg})}.
	\]
	Also $\big(\sum_D s_D(e_w,e_w)\big)^{1/2}\le |||e_h|||$. Hence
	\begin{equation}\label{eq:J3_bound}
		|J_3|\lesssim h^k\|u\|_{H^{k+1}(\Omega)}\,|||e_h||| .
	\end{equation}
	
	\medskip\noindent
	\textbf{Estimate of $J_4$.}
	By Cauchy--Schwarz and the interpolation estimate \eqref{eq:approx_Ih},
	\[
	|J_4|
	\le \Big(\sum_T \|\nabla(I_hu-u)\|_{L^2(T)}^2\Big)^{1/2}
	\Big(\sum_T \|\nabla e_c\|_{L^2(T)}^2\Big)^{1/2}
	\lesssim h^k\|u\|_{H^{k+1}(\Omega_{cg})}\,|||e_h|||
	\le h^k\|u\|_{H^{k+1}(\Omega)}\,|||e_h||| .
	\]
	Thus,
	\begin{equation}\label{eq:J4_bound}
		|J_4|\lesssim h^k\|u\|_{H^{k+1}(\Omega)}\,|||e_h||| .
	\end{equation}
	
	\medskip\noindent
	\textbf{Conclusion.}
	Combining \eqref{eq:J1_bound}--\eqref{eq:J4_bound} yields
	\[
	|||e_h|||^2 \;\le\; C\, h^k\|u\|_{H^{k+1}(\Omega)}\,|||e_h||| .
	\]
	If $|||e_h|||=0$ the claim is trivial; otherwise divide both sides by $|||e_h|||$ to obtain
	\eqref{eq:energy_est}. This completes the proof.
\end{proof}

\subsection{$L^2$ Norm Estimate}\label{subsec:L2}

For $e_h=(e_c,e_w)\in V_h$ with $e_w=(e_0,e_b)$, define the piecewise scalar error
\[
e_h^*(x):=
\begin{cases}
	e_0(x), & x\in \Omega_{wg},\\
	e_c(x), & x\in \Omega_{cg}.
\end{cases}
\qquad\text{so that}\qquad
\|e_h^*\|_{L^2(\Omega)}^2=\|e_0\|_{L^2(\Omega_{wg})}^2+\|e_c\|_{L^2(\Omega_{cg})}^2 .
\]
 {The energy-norm estimate above uses only the geometric assumptions on the mesh and local approximation properties. To derive the $L^2$ estimate, we additionally assume} the usual elliptic regularity on the curved domain: for any $\Psi\in L^2(\Omega)$, the solution
$\Phi\in H_0^1(\Omega)$ of $-\Delta\Phi=\Psi$ satisfies
\begin{equation}\label{eq:dual_reg}
	\|\Phi\|_{H^2(\Omega)} \lesssim \|\Psi\|_{L^2(\Omega)}.
\end{equation}
 {A sufficient condition for this elliptic regularity is that the boundary is $C^{1,1}$.} 
\begin{theorem}\label{thm:L2_norm}
	Assume \eqref{eq:dual_reg}. Let $u\in H^{k+1}(\Omega)\cap H_0^1(\Omega)$ solve $-\Delta u=f$ in $\Omega$
	and let $u_h\in V_h$ solve \eqref{eq:scheme}. Let $e_h=\Pi_hu-u_h$ and define $e_h^*$ as above.
	Then there exists $C>0$, independent of $h$, such that
	\begin{equation}\label{eq:L2_est}
		\|e_h^*\|_{L^2(\Omega)} \le C\, h^{k+1}\,\|u\|_{H^{k+1}(\Omega)}.
	\end{equation}
\end{theorem}

\begin{proof}
	Let $\Phi\in H_0^1(\Omega)\cap H^2(\Omega)$ solve the dual problem
	\begin{equation}\label{eq:dual_problem}
		-\Delta\Phi=e_h^* \quad\text{in }\Omega,\qquad \Phi=0\quad\text{on }\partial\Omega.
	\end{equation}
	Let $\Pi_h\Phi\in V_h$ be the global interpolation--projection defined by
	$\Pi_h\Phi|_{\Omega_{cg}}=I_h\Phi$ and $\Pi_h\Phi|_D=Q_h\Phi=(Q_{k,D}^0\Phi,Q_{b,D}^\partial\Phi)$ on
	each $D\in\mathcal T_h^{wg}$.
	
	\medskip\noindent
	\textbf{Step 1: A duality identity $ \|e_h^*\|^2 = A_h(e_h,\Pi_h\Phi) + \mathcal R(\Phi;e_h)$.}
	By \eqref{eq:dual_problem},
	\[
	\|e_h^*\|_{L^2(\Omega)}^2 = (-\Delta\Phi,e_0)_{\Omega_{wg}}+ (-\Delta\Phi,e_c)_{\Omega_{cg}}.
	\]
	Integrating by parts elementwise on $\Omega_{wg}$ and once on $\Omega_{cg}$ yields
	\begin{align}\label{eq:dual_ibp}
		\|e_h^*\|_{L^2(\Omega)}^2
		&=\sum_{D\in\mathcal T_h^{wg}}(\nabla\Phi,\nabla e_0)_D
		-\sum_{D\in\mathcal T_h^{wg}}\langle \nabla\Phi\cdot \mathbf n_{wg},\, e_0\rangle_{\partial D}
		+(\nabla\Phi,\nabla e_c)_{\Omega_{cg}}
		-\langle \nabla\Phi\cdot \mathbf n_{cg},\, e_c\rangle_{\Gamma}.
	\end{align}
	Rewrite $\langle\nabla\Phi\cdot \mathbf n_{wg},e_0\rangle_{\partial D}
	=\langle\nabla\Phi\cdot \mathbf n_{wg},e_b\rangle_{\partial D}
	+\langle\nabla\Phi\cdot \mathbf n_{wg},e_0-e_b\rangle_{\partial D}$ and sum over $D$.
	The terms with $e_b$ cancel on interior WG edges and vanish on $\partial\Omega$; thus
	$\sum_D\langle\nabla\Phi\cdot \mathbf n_{wg},e_b\rangle_{\partial D}
	=\langle\nabla\Phi\cdot \mathbf n_{wg},e_b\rangle_{\Gamma}$.
	Since $\mathbf n_{wg}=-\mathbf n_{cg}$ on $\Gamma$ and $e_b=e_c$ on $\Gamma$ (because $e_h\in V_h$),
	the interface fluxes cancel:
	$$
	\langle\nabla\Phi\cdot\mathbf n_{wg},e_b\rangle_{\Gamma}
	+\langle\nabla\Phi\cdot\mathbf n_{cg},e_c\rangle_{\Gamma}=0.
	$$
	Therefore \eqref{eq:dual_ibp} reduces to
	\begin{equation}\label{eq:dual_reduced}
		\|e_h^*\|_{L^2(\Omega)}^2
		=\sum_{D}(\nabla\Phi,\nabla e_0)_D
		-\sum_{D}\langle \nabla\Phi\cdot \mathbf n_{wg},\, e_0-e_b\rangle_{\partial D}
		+(\nabla\Phi,\nabla e_c)_{\Omega_{cg}}.
	\end{equation}
	Now fix $D\in\mathcal T_h^{wg}$ and write $\mathbb Q_{k-1,D}$ for the $L^2(D)$-projection onto
	$[\mathbb P_{k-1}(D)]^2$.
	Split $\nabla\Phi=\mathbb Q_{k-1,D}\nabla\Phi+(\nabla\Phi-\mathbb Q_{k-1,D}\nabla\Phi)$ and use the identity
	(valid for any $\vec q\in[\mathbb P_{k-1}(D)]^2$)
	\begin{equation}\label{eq:wg_key_identity}
		(\nabla_w e_w,\vec q)_D = (\nabla e_0,\vec q)_D-\langle e_0-e_b, \vec q\cdot \mathbf n_{wg}\, \rangle_{\partial D},
	\end{equation}
	which follows directly from the definition of $\nabla_w$ by integrating $(e_0,\nabla\!\cdot \vec q)_D$ by parts.
	
	\noindent Taking $\vec q=\mathbb Q_{k-1,D}\nabla\Phi$ in \eqref{eq:wg_key_identity} and inserting into
	\eqref{eq:dual_reduced} gives
	\begin{align}\label{eq:dual_decomp_v1}
		\|e_h^*\|_{L^2(\Omega)}^2
		&=\sum_{D}(\nabla_w e_w,\mathbb Q_{k-1,D}\nabla\Phi)_D
		+\sum_{D}(\nabla e_0,\nabla\Phi-\mathbb Q_{k-1,D}\nabla\Phi)_D\nonumber\\
		&\quad -\sum_{D}\langle (\nabla\Phi-\mathbb Q_{k-1,D}\nabla\Phi)\cdot \mathbf n_{wg},\, e_0-e_b\rangle_{\partial D}\nonumber\\
		&\quad +(\nabla I_h\Phi,\nabla e_c)_{\Omega_{cg}}
		+(\nabla(\Phi-I_h\Phi),\nabla e_c)_{\Omega_{cg}}.
	\end{align}
	Since $e_0|_D \in \mathbb{P}_k(D)$, its gradient satisfies $\nabla e_0 \in [\mathbb{P}_{k-1}(D)]^2$. By the definition of the $L^2$ projection $\mathbb{Q}_{k-1,D}$, the second term vanishes identically:
	\begin{equation}\label{eq:vanish_term}
		\sum_{D}(\nabla e_0,\nabla\Phi-\mathbb Q_{k-1,D}\nabla\Phi)_D = 0.
	\end{equation}
	Next, apply Lemma~\ref{lem:I-P-estimation} (identity \eqref{eq:lem1}) to $\xi=\Phi$ with
	$\vec q=\nabla_w e_w\in[\mathbb P_{k-1}(D)]^2$:
	\[
	(\nabla_w Q_h\Phi,\nabla_w e_w)_D
	=(\mathbb Q_{k-1,D}\nabla\Phi,\nabla_w e_w)_D
	+\langle Q_{b,D}^\partial\Phi-\Phi,\, \nabla_w e_w\cdot \mathbf n_{wg}\rangle_{\partial D},
	\]
	hence
	\[
	(\nabla_w e_w,\mathbb Q_{k-1,D}\nabla\Phi)_D
	=(\nabla_w e_w,\nabla_w Q_h\Phi)_D
	-\langle Q_{b,D}^\partial\Phi-\Phi,\, \nabla_w e_w\cdot \mathbf n_{wg}\rangle_{\partial D}.
	\]
	Insert this into \eqref{eq:dual_decomp_v1} and add/subtract the stabilization term to identify $A_h$:
	\begin{equation}\label{eq:duality_identity_final}
		\|e_h^*\|_{L^2(\Omega)}^2
		= A_h(e_h,\Pi_h\Phi) + \mathcal R(\Phi;e_h),
	\end{equation}
	where the remainder $\mathcal R(\Phi;e_h)$ drops the zero term \eqref{eq:vanish_term} and is simply given by
	\begin{align}\label{eq:R_def_v1}
		\mathcal R(\Phi;e_h)
		&:= -\sum_{D}\langle (\nabla\Phi-\mathbb Q_{k-1,D}\nabla\Phi)\cdot \mathbf n_{wg},\, e_0-e_b\rangle_{\partial D}\nonumber\\
		&\quad -\sum_{D}\langle Q_{b,D}^\partial\Phi-\Phi,\, \nabla_w e_w\cdot \mathbf n_{wg}\rangle_{\partial D} \nonumber \\
		&\quad +(\nabla(\Phi-I_h\Phi),\nabla e_c)_{\Omega_{cg}}
		-\sum_{D}s_D(e_w,Q_h\Phi).
	\end{align}
	
	\medskip\noindent
	\textbf{Step 2: Bound the remainder $\mathcal R(\Phi;e_h)$.}
	We claim
	\begin{equation}\label{eq:R_bound}
		|\mathcal R(\Phi;e_h)| \lesssim h\,\|\Phi\|_{H^2(\Omega)}\,|||e_h||| .
	\end{equation}
	Because $\Phi \in H^2(\Omega)$, by Lemmas \ref{lem:local_approx}--\ref{lem:stab_consistency} and \eqref{err:u-QbDpartial-u}, we have the following approximation estimates:
	\begin{align*}
	&\|(\nabla\Phi-\mathbb Q_{k-1,D}\nabla\Phi)\cdot \mathbf n_{wg}\|_{L^2(\partial D)}\lesssim h_D^{1/2}\|\Phi\|_{H^2(D)},\\
	&\|Q_{b,D}^\partial\Phi-\Phi\|_{L^2(\partial D)}\lesssim h_D^{3/2}\|\Phi\|_{H^2(D)}, \\
	&s_D(Q_h\Phi,Q_h\Phi)^{1/2}\lesssim h_D\|\Phi\|_{H^2(D)},\\
	&\|\nabla(\Phi-I_h\Phi)\|_{L^2(\Omega_{cg})}\lesssim h\|\Phi\|_{H^2(\Omega)}.
	\end{align*}
	Using these facts, each term in \eqref{eq:R_def_v1} is bounded as follows:
	\begin{itemize}
		\item For the first term, we write $h_D^{-1/2}\|e_0-e_b\|_{L^2(\partial D)} = s_D(e_w, e_w)^{1/2}$ and bound it by the energy norm.
		\item For the second term, apply the inverse trace inequality for polynomials to get $\|\nabla_w e_w \cdot \mathbf n_{wg}\|_{L^2(\partial D)} \lesssim h_D^{-1/2} \|\nabla_w e_w\|_{L^2(D)}$ and multiply by the $h_D^{3/2}$ bound.
		\item The third and fourth terms are directly bounded using Cauchy--Schwarz and the definition of the energy norm.
	\end{itemize}
	Summing over all elements yields $|\mathcal{R}(\Phi; e_h)| \le C h\|\Phi\|_{H^2(\Omega)}|||e_h|||$. 
	
	\medskip\noindent
	\textbf{Step 3: Bound $A_h(e_h,\Pi_h\Phi)$ using the error equation and interface cancellation.}
	Apply the error equation \eqref{eq:error_eq_1} with $v=\Pi_h\Phi$ to obtain
	\begin{align}\label{eq:Ah_error}
		A_h(e_h,\Pi_h\Phi)
		&=\sum_{D}\Big\langle (\nabla u-\mathbb Q_{k-1,D}\nabla u)\cdot \mathbf n_{wg},\, Q_{k,D}^0\Phi-Q_{b,D}^\partial\Phi\Big\rangle_{\partial D}\nonumber\\
		&\quad +\sum_{D}\Big\langle Q_{b,D}^\partial u-u,\, \nabla_w Q_h\Phi\cdot \mathbf n_{wg}\Big\rangle_{\partial D}\nonumber\\
		&\quad +\sum_{D}s_D(Q_hu,Q_h\Phi)
		+(\nabla I_hu-\nabla u,\nabla I_h\Phi)_{\Omega_{cg}}.
	\end{align}
	The first and third terms are estimated directly:
	\begin{align*}
		\Big|\sum_{D}\langle (\nabla u-\mathbb Q_{k-1,D}\nabla u)\cdot \mathbf n_{wg},\, Q_{k,D}^0\Phi-Q_{b,D}^\partial\Phi\rangle_{\partial D}\Big| 
		&\lesssim \sum_D (h_D^{k-1/2}\|u\|_{H^{k+1}(D)})(h_D^{3/2}\|\Phi\|_{H^2(D)}) \\
		&\lesssim h^{k+1}\|u\|_{H^{k+1}(\Omega)}\|\Phi\|_{H^2(\Omega)},
	\end{align*}
	and
	\begin{align*}
		\Big|\sum_{D}s_D(Q_hu,Q_h\Phi)\Big|
		&\le \Big(\sum_D s_D(Q_hu,Q_hu)\Big)^{1/2}\Big(\sum_D s_D(Q_h\Phi,Q_h\Phi)\Big)^{1/2}\\
		&\lesssim (h^k\|u\|_{H^{k+1}(\Omega)})(h\|\Phi\|_{H^2(\Omega)})
		\lesssim h^{k+1}\|u\|_{H^{k+1}(\Omega)}\|\Phi\|_{H^2(\Omega)}.
	\end{align*}
	For the remaining two terms in \eqref{eq:Ah_error}, we \emph{do not} bound them separately; instead, we
	combine them to exploit cancellation on the interface $\Gamma$.
	
	\noindent \emph{(a) Reduce the WG boundary term to $\Gamma$.}
	We split the boundary $\partial D_{proj}$ into straight outer-inner edges (which are internal to the domain) and curved outer edges (which lie on $\partial\Omega$). On any straight edge $e\subset\partial D_{proj} \setminus \partial\Omega$, the normal vector $\mathbf n_{wg}$ is constant. Since $\nabla_w Q_h\Phi \in[\mathbb P_{k-1}(D)]^2$, its normal trace $\nabla_w Q_h\Phi\cdot \mathbf n_{wg}|_e$ is a polynomial in $\mathbb P_{k-1}(e)\subset \mathbb P_k(e)$. Because $Q_{b,D}^\partial$ is the $L^2(e)$-projection onto $\mathbb P_k(e)$, by orthogonality we have:
	\[
	\langle Q_{b,D}^\partial u-u,\ \nabla_w Q_h\Phi\cdot \mathbf n_{wg}\rangle_e=0,\qquad e\subset\partial D_{proj} \setminus \partial\Omega.
	\]
	On the other hand, for any curved edge $e \subset \partial D_{proj} \cap \partial\Omega$, the homogeneous Dirichlet boundary condition requires $u=0$. Consequently, $Q_{b,D}^\partial u = 0$, making the integral trivially zero on $\partial \Omega$. 
	
	\noindent Hence, the only non-vanishing contribution comes from the interface $\Gamma$ (where $e \subset \partial D_{int}$):
	\begin{equation}\label{eq:wg_term_on_Gamma}
		\sum_{D}\Big\langle Q_{b,D}^\partial u-u,\ \nabla_w Q_h\Phi\cdot \mathbf n_{wg}\Big\rangle_{\partial D}
		=\Big\langle I_hu-u,\ \nabla_w Q_h\Phi\cdot \mathbf n_{wg}\Big\rangle_{\Gamma},
	\end{equation}
	because on $\Gamma$ we defined $Q_{b,D}^\partial u=I_{k,e}^\partial u=I_hu|_\Gamma$.
	
	\noindent \emph{(b) Split the CG term and integrate by parts to produce a cancelling interface flux.}
	Write
	\begin{align}\label{eq:cg_split_v1}
		(\nabla I_hu-\nabla u,\nabla I_h\Phi)_{\Omega_{cg}}
		&=(\nabla I_hu-\nabla u,\nabla(I_h\Phi-\Phi))_{\Omega_{cg}}
		+(\nabla I_hu-\nabla u,\nabla\Phi)_{\Omega_{cg}}.
	\end{align}
	The first term is bounded by interpolation:
	\[
	|(\nabla I_hu-\nabla u,\nabla(I_h\Phi-\Phi))_{\Omega_{cg}}|
	\lesssim (h^k\|u\|_{H^{k+1}(\Omega)})(h\|\Phi\|_{H^2(\Omega)})
	\lesssim h^{k+1}\|u\|_{H^{k+1}(\Omega)}\|\Phi\|_{H^2(\Omega)}.
	\]
	For the second term, integrate by parts on $\Omega_{cg}$ (whose outward normal is $\mathbf n_{cg}=-\mathbf n_{wg}$):
	\begin{align}\label{eq:cg_ibp_v1}
		(\nabla I_hu-\nabla u,\nabla\Phi)_{\Omega_{cg}}
		&=-(I_hu-u,\Delta\Phi)_{\Omega_{cg}}+\langle I_hu-u,\nabla\Phi\cdot \mathbf n_{cg}\rangle_{\Gamma}\nonumber\\
		&=(I_hu-u,e_h^*)_{\Omega_{cg}}-\langle I_hu-u,\nabla\Phi\cdot \mathbf n_{wg}\rangle_{\Gamma},
	\end{align}
	where we used $-\Delta\Phi=e_h^*$ and $\mathbf n_{cg}=-\mathbf n_{wg}$.
	
	Now combine the interface flux in \eqref{eq:cg_ibp_v1} with \eqref{eq:wg_term_on_Gamma}:
	\begin{align}\label{eq:interface_cancellation}
		\Big\langle I_hu-u,\ \nabla_w Q_h\Phi\cdot \mathbf n_{wg}\Big\rangle_{\Gamma}
		-\Big\langle I_hu-u,\ \nabla\Phi\cdot \mathbf n_{wg}\Big\rangle_{\Gamma}
		=\Big\langle I_hu-u,\ (\nabla_w Q_h\Phi-\nabla\Phi)\cdot \mathbf n_{wg}\Big\rangle_{\Gamma}.
	\end{align}
	Using the inverse trace inequality and Lemma~\ref{lem:I-P-estimation} (estimate \eqref{eq:lem3} applied to
	$\xi=\Phi\in H^2$ gives $\|\nabla_w Q_h\Phi-\mathbb Q_{k-1,D}\nabla\Phi\|_{L^2(D)}\lesssim h_D\|\Phi\|_{H^2(D)}$),
	together with $\|\nabla\Phi-\mathbb Q_{k-1,D}\nabla\Phi\|_{L^2(D)}\lesssim h_D\|\Phi\|_{H^2(D)}$, we obtain
	\begin{equation}\label{eq:flux_mismatch}
		\|(\nabla_w Q_h\Phi-\nabla\Phi)\cdot \mathbf n_{wg}\|_{L^2(\Gamma)}
		\lesssim h^{1/2}\|\Phi\|_{H^2(\Omega)}.
	\end{equation}
	Therefore, using $\|I_hu-u\|_{L^2(\Gamma)}\lesssim h^{k+1/2}\|u\|_{H^{k+1}(\Omega)}$,
	\begin{equation}\label{eq:interface_bound}
		\Big|\langle I_hu-u,\ (\nabla_w Q_h\Phi-\nabla\Phi)\cdot \mathbf n_{wg}\rangle_{\Gamma}\Big|
		\lesssim h^{k+1}\|u\|_{H^{k+1}(\Omega)}\|\Phi\|_{H^2(\Omega)}.
	\end{equation}
	Finally, from \eqref{eq:cg_ibp_v1},
	\begin{equation}\label{eq:bulk_L2_term}
		|(I_hu-u,e_h^*)_{\Omega_{cg}}|
		\le \|I_hu-u\|_{L^2(\Omega_{cg})}\,\|e_h^*\|_{L^2(\Omega)}
		\lesssim h^{k+1}\|u\|_{H^{k+1}(\Omega)}\,\|e_h^*\|_{L^2(\Omega)}.
	\end{equation}
	Collecting all bounds for the four terms in \eqref{eq:Ah_error} (using \eqref{eq:wg_term_on_Gamma},
	\eqref{eq:cg_split_v1}--\eqref{eq:bulk_L2_term}), we obtain
	\begin{equation}\label{eq:Ah_bound}
		|A_h(e_h,\Pi_h\Phi)|
		\lesssim h^{k+1}\|u\|_{H^{k+1}(\Omega)}\Big(\|\Phi\|_{H^2(\Omega)}+\|e_h^*\|_{L^2(\Omega)}\Big).
	\end{equation}
	
	\medskip\noindent
	\textbf{Step 4: Conclude the $L^2$ estimate.}
	Insert \eqref{eq:Ah_bound} and \eqref{eq:R_bound} into \eqref{eq:duality_identity_final}:
	\[
	\|e_h^*\|_{L^2(\Omega)}^2
	\lesssim h^{k+1}\|u\|_{H^{k+1}(\Omega)}\Big(\|\Phi\|_{H^2(\Omega)}+\|e_h^*\|_{L^2(\Omega)}\Big)
	+ h\,\|\Phi\|_{H^2(\Omega)}\,|||e_h||| .
	\]
	Using the energy estimate $|||e_h|||\lesssim h^k\|u\|_{H^{k+1}(\Omega)}$ (Theorem~\ref{thm:energy_norm})
	and the dual regularity \eqref{eq:dual_reg} with $\Psi=e_h^*$ (so $\|\Phi\|_{H^2(\Omega)}\lesssim \|e_h^*\|_{L^2(\Omega)}$),
	we arrive at
	\[
	\|e_h^*\|_{L^2(\Omega)}^2
	\lesssim h^{k+1}\|u\|_{H^{k+1}(\Omega)}\,\|e_h^*\|_{L^2(\Omega)}.
	\]
	If $\|e_h^*\|_{L^2(\Omega)}=0$ we are done; otherwise divide by $\|e_h^*\|_{L^2(\Omega)}$ to obtain
	\eqref{eq:L2_est}.
\end{proof}

\section{Numerical Experiments}\label{sec:numerics}

This section verifies the convergence theory for the coupled WG--CG method on curved domains.
All computations are performed on a unit disk
\[
\Omega=\{(x,y)\in\mathbb R^2:\ x^2+y^2<1\},
\]
with a boundary-layer mesh: a collection of curvilinear WG elements adjacent to $\partial\Omega$,
and a polygonal interior region triangulated for the conforming CG method.

\medskip
\noindent
{\bf Meshes.}
Let $R_{\mathrm{out}}$ and $R_{\mathrm{in}}$ denote the outer and inner radii of the boundary layer, respectively. Since the computational domain is the unit disk, we set $R_{\mathrm{out}}=1$ and $R_{\mathrm{in}}=1-H$, so that
\[
H:=R_{\mathrm{out}}-R_{\mathrm{in}}
\]
is exactly the boundary-layer thickness. The interface $\Gamma$ is a
regular $N_\theta$-gon with
\[
N_\theta \approx \left\lfloor \frac{2\pi}{H} \right\rfloor,
\]
generated by connecting vertices on the circle of radius $R_{\mathrm{in}}$.
Each WG element $D$ is bounded by one outer circular arc on $\partial\Omega$, two radial segments, and one inner chord
(edge on $\Gamma$). The interior $\Omega_{cg}$ is triangulated by a Delaunay triangulation of the interior nodes. For each refinement, $H$ is decreased by half, $H$ starts from 0.1, see Figure \ref{fig:P1-mesh}.

\begin{figure}[h!]
	\centering
	\includegraphics[width=0.4\textwidth]{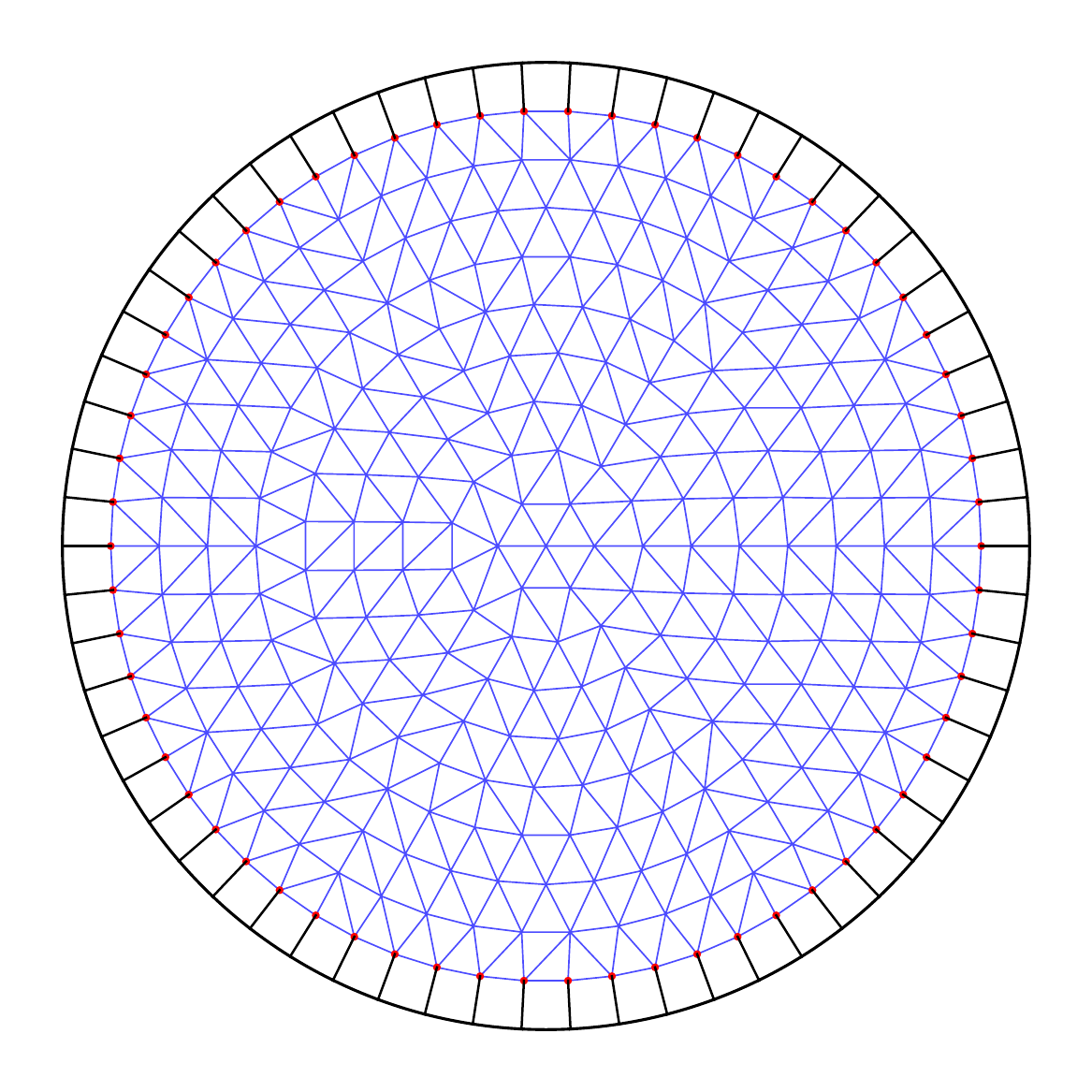}
	\includegraphics[width=0.4\textwidth]{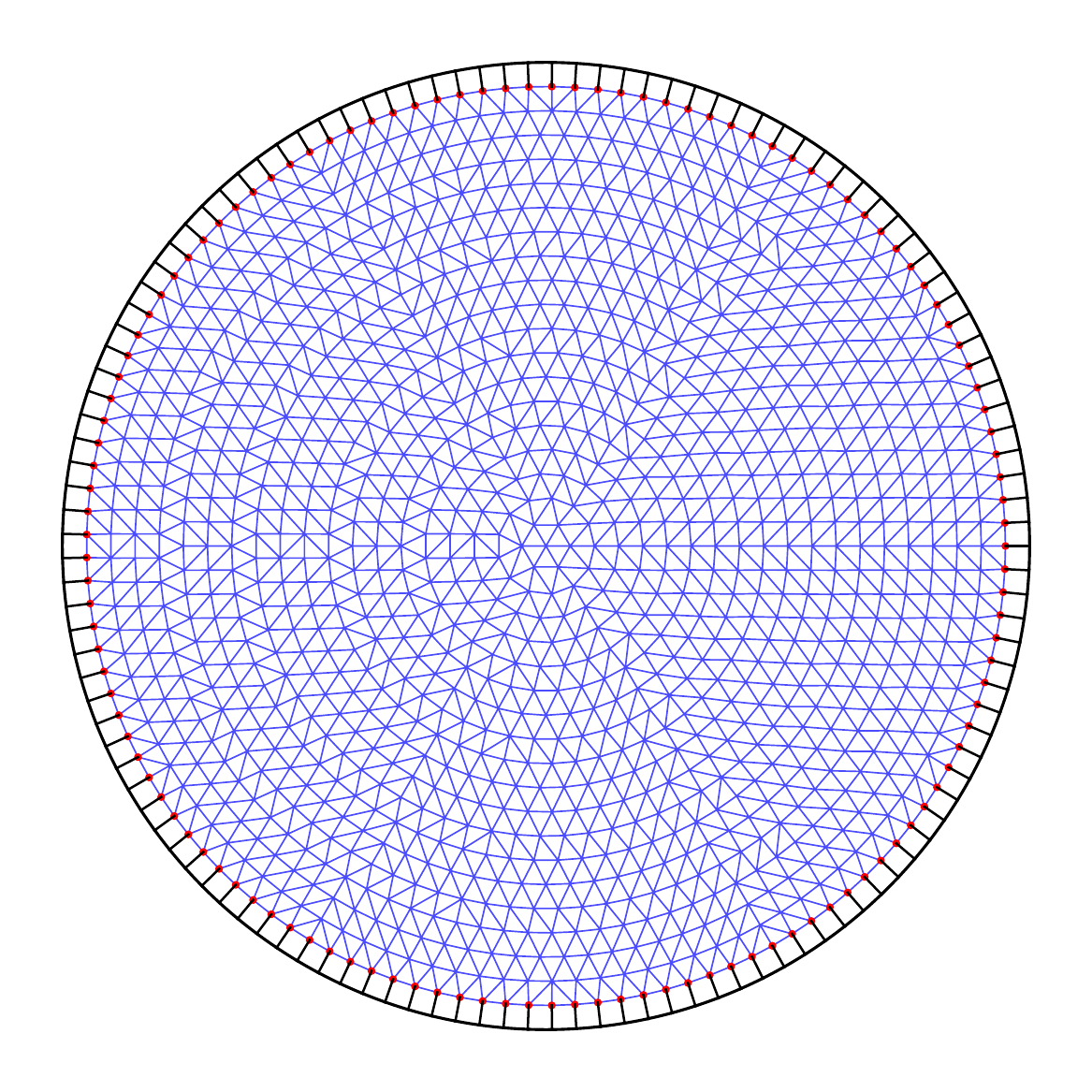}
	\caption{The coupled mesh for Continuous Galerkin and Weak Galerkin. {\bf Left:} a coarse mesh, $H = 0.1$. {\bf Right:} a refined mesh, $H = 0.05$. Red dots are interface points. The numerical solution is continuous in the CG region (triangles) but discontinuous in the WG region (curvilinear polygons). As the mesh is refined, the boundary layer becomes thinner. }
	\label{fig:P1-mesh}
\end{figure}

\medskip
\noindent
{\bf Error metrics and empirical rates.}  
Let $u$ be the exact solution and $u_h=(u_c,u_w)\in V_h$ be the discrete solution.
We report the energy-norm error
\[
|||\,\Pi_h u-u_h\,|||^2
=\sum_{D\in\mathcal T_h^{wg}}
\Big(\|\nabla_w(\Pi_hu-u_h)\|_{L^2(D)}^2+s_D(\Pi_hu-u_h,\Pi_hu-u_h)\Big)
+\sum_{T\in\mathcal T_h^{cg}}\|\nabla(\Pi_hu-u_h)\|_{L^2(T)}^2,
\]
and the $L^2$-error of the piecewise scalar field
\[
u_h^*(x)=
\begin{cases}
	(u_h)_0(x), & x\in\Omega_{wg},\\
	u_c(x), & x\in\Omega_{cg},
\end{cases}
\quad
\text{so that}\quad
\|u-u_h^*\|_{L^2(\Omega)}^2=\|u-(u_h)_0\|_{L^2(\Omega_{wg})}^2+\|u-u_c\|_{L^2(\Omega_{cg})}^2.
\]
Given a refinement sequence $\{H_\ell\}$ and the errors $\{E_{H_\ell}\}$, the empirical rate is computed by
\[
\mathrm{rate}=\frac{\log(E_{H_\ell}/E_{H_{\ell+1}})}{\log(H_\ell/H_{\ell+1})}
\quad\text{(with } H_{\ell+1}<H_\ell\text{)}.
\]
The theory predicts
\[
|||\,\Pi_hu-u_h\,|||\;=\;\mathcal O(h^k),
\qquad
\|u-u_h^*\|_{L^2(\Omega)}\;=\;\mathcal O(h^{k+1}).
\]
We confirm these rates for $k=1$ ($P_1$ element) and $k=2$ ($P_2$ element). 

\begin{figure}[h!]
	\centering
	\includegraphics[width=0.48\textwidth]{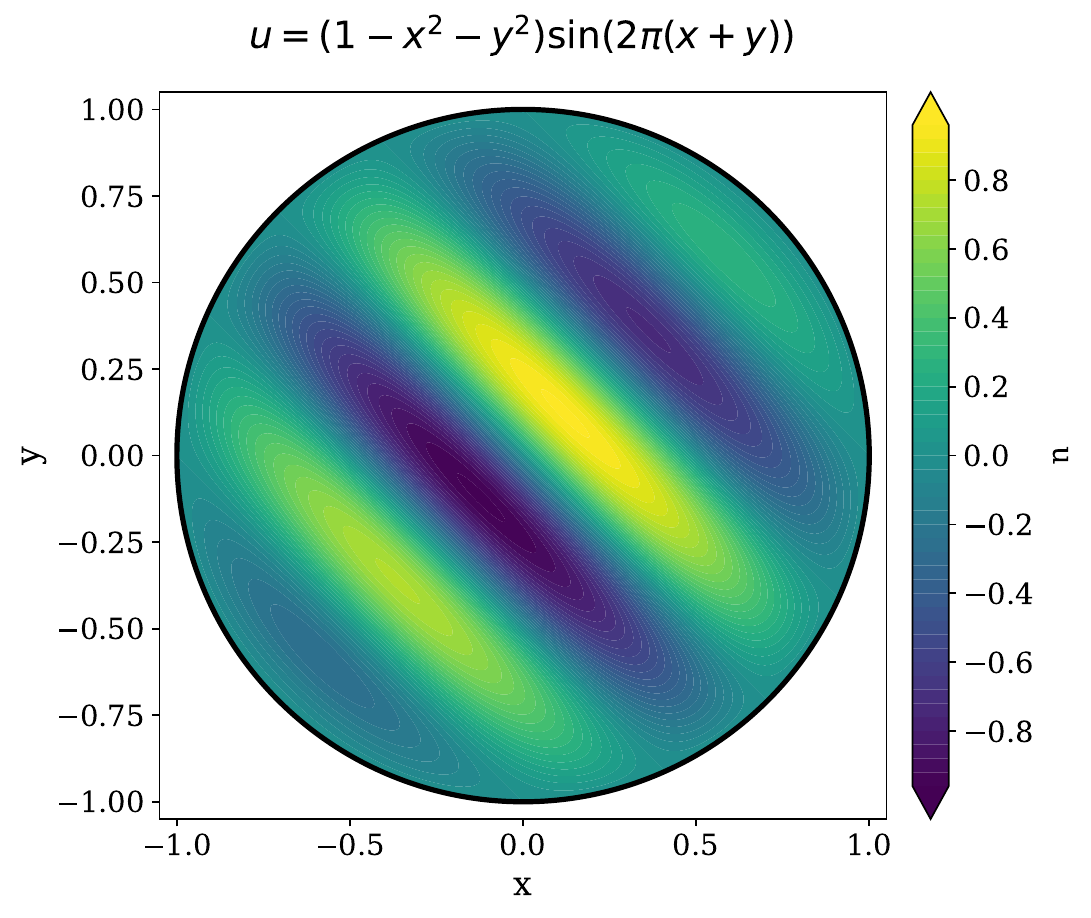}
	\includegraphics[width=0.48\textwidth]{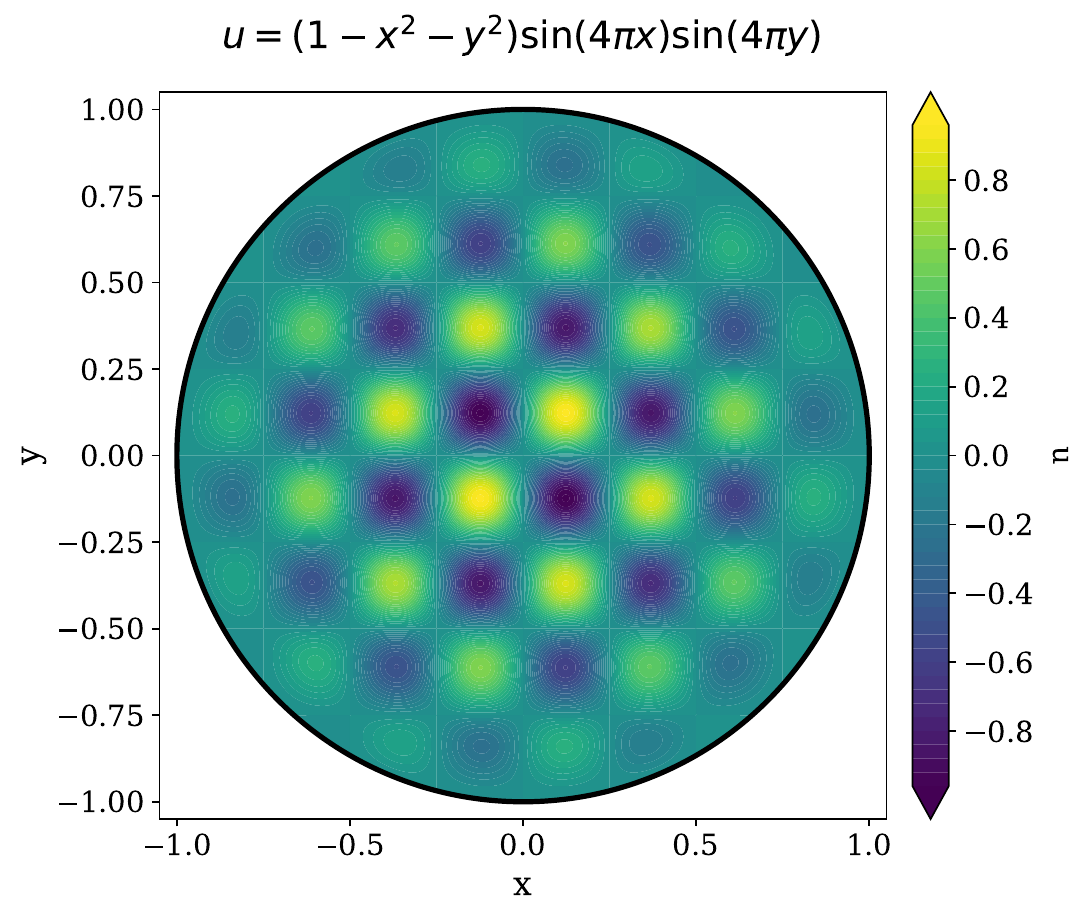}
	\caption{{\bf Left:} graph of the exact solution for Test 1. {\bf Right:} graph of the exact solution for Test 2. In comparison with Test 1, the solution for Test 2 exhibits higher oscillatory behavior.}
	\label{fig:u-1-2}
\end{figure}
\subsection{Test 1: smooth manufactured solution on the unit disk}\label{subsec:test1}

We choose the exact solution
\[
u(x, y) = (1 - x^2 - y^2) \sin(2\pi(x + y)),
\quad u|_{\partial\Omega}=0,
\]
with $f=-\Delta u$ computed analytically, see the graph of $u$ in the left of Figure \ref{fig:u-1-2}. Numerical results are given in Table \ref{tab:test1_k1} and Table \ref{tab:test1_k2}.

\begin{table}[h!]\centering
	\caption{Test 1 -- Convergence Rate for $P_1$ Elements}
	\label{tab:test1_k1}
	\begin{tabular}{c|c|c|c|c}
		\hline
		$H$ & $|||\,\Pi_hu-u_h\,|||$ & rate & $\|u-u_h^*\|_{L^2(\Omega)}$ & rate \\ \hline
$1/10 $ & 6.8137e-01      & -    & 6.5392e-02      & - \\ \hline
$1/20 $ & 2.5475e-01      & 1.42 & 1.6092e-02      & 2.02 \\ \hline
$1/40 $ & 9.6410e-02      & 1.40 & 3.8834e-03      & 2.05 \\ \hline
$1/80 $ & 3.5411e-02      & 1.45 & 9.4701e-04      & 2.04 \\ \hline
$1/160$ & 1.2741e-02      & 1.47 & 2.3242e-04      & 2.03 \\ \hline	 
	\end{tabular}
\end{table}
\begin{table}[h!]\centering
	\caption{Test 1 -- Convergence Rate for $P_2$ Elements}
	\label{tab:test1_k2}
	\begin{tabular}{c|c|c|c|c}
		\hline
$H$ & $|||\Pi_hu-u_h|||$ & rate & $\|u-u_h^*\|_{L^2(\Omega)}$ & rate \\ \hline
$1/10 $ & 2.2602e-01      & - & 8.9614e-03      & - \\ \hline
$1/20 $ & 4.1040e-02      & 2.46 & 8.5524e-04      & 3.39 \\ \hline
$1/40 $ & 7.2545e-03      & 2.50 & 8.3467e-05      & 3.36 \\ \hline
$1/80 $ & 1.2939e-03      & 2.49 & 8.6709e-06      & 3.27 \\ \hline
$1/160$ & 2.3504e-04      & 2.46 & 9.6004e-07      & 3.18 \\ \hline	 
	\end{tabular}
\end{table}

\subsection{Test 2: oscillatory manufactured solution}\label{subsec:test2}

To further challenge the discretization, we consider
\[
u(x, y) = (1 - x^2 - y^2) \sin(4\pi x)\sin(4\pi y),
\quad u|_{\partial\Omega}=0,
\]
with $f=-\Delta u$ computed analytically.
This test contains stronger oscillatory, see the right in Figure \ref{fig:u-1-2}. Numerical results are given in Table \ref{tab:test2_k1} and Table \ref{tab:test2_k2}.

\begin{table}[h!]\centering
	\caption{Test 2 -- Convergence Rate for $P_1$ Elements}
	\label{tab:test2_k1}
	\begin{tabular}{c|c|c|c|c}
	\hline
$H$ & $|||\Pi_hu-u_h|||$ & rate & $\|u-u_h^*\|_{L^2(\Omega)}$ & rate \\ \hline
$1/10 $ & 1.0826e+00      &  -  & 1.3547e-01      &  -  \\ \hline
$1/20 $ & 4.7025e-01      & 1.20 & 3.7598e-02      & 1.85 \\ \hline
$1/40 $ & 1.9260e-01      & 1.29 & 9.5820e-03      & 1.97 \\ \hline
$1/80 $ & 7.0569e-02      & 1.45 & 2.3803e-03      & 2.01 \\ \hline
$1/160$ & 2.5129e-02      & 1.49 & 5.8988e-04      & 2.01 \\ \hline
	\end{tabular}
\end{table}

\begin{table}[h!]\centering
	\caption{Test 2 -- Convergence Rate for $P_2$ Elements}
	\label{tab:test2_k2}
	\begin{tabular}{c|c|c|c|c}
		\hline
$H$ & $|||\Pi_hu-u_h|||$ & rate & $\|u-u_h^*\|_{L^2(\Omega)}$ & rate \\ \hline
$1/10 $ & 6.2742e-01      &  -  & 2.5734e-02      &  -  \\ \hline
$1/20 $ & 1.0211e-01      & 2.62 & 2.5141e-03      & 3.36 \\ \hline
$1/40 $ & 1.7490e-02      & 2.55 & 2.7360e-04      & 3.20 \\ \hline
$1/80 $ & 3.2549e-03      & 2.43 & 3.2231e-05      & 3.09 \\ \hline
$1/160$ & 6.4737e-04      & 2.33 & 3.9131e-06      & 3.04 \\ \hline
	\end{tabular}
\end{table}

\subsection{DoF reduction compared to a fully nonconforming method}\label{subsec:dof_reduction}

A key motivation for the coupled scheme is to avoid nonconforming unknowns in the interior.
Although static condensation via the Schur complement and polynomial reduction are available in the literature, all degrees of freedom are counted here for a fair comparison.
For $k=1$ the coupled method introduces, per WG element, $3$ interior coefficients for $v_0$ and per radial edge
$2$ trace dofs for $v_b$, while the interface dofs are shared with CG.
For $k=2$ the counts become $6$ interior coefficients per WG element and $3$ trace dofs per radial edge. For $k=3,4$, the dofs can also be calculated similarly. 
To compare the degrees of freedom for CG and the coupled CG-WG scheme, we only consider dofs of CG in the interior region, the results are given in Table \ref{tab:dof_compare_1}. 
To quantify savings, we report the ratio of total number of degrees of freedom for the coupled CG-WG scheme over dofs of full WG discretization on the same mesh, which is shown in Table \ref{tab:dof_compare_2}. 

\begin{table}[h!]\centering
	\caption{Ratio $\frac{\text{CG dofs}}{\text{WG-CG dofs}}$ for $P_k$ Elements}
	\label{tab:dof_compare_1}
	\begin{tabular}{c|c|c|c|c|c|c}
		\hline
$H$&CG Element     &WG Element   &$P_1$     &$P_2$     &$P_3$     &$P_4$      \\ \hline
$1/10 $ & 514        & 62         & 48.25\% & 66.16\% & 73.50\% & 77.36\% \\ \hline
$1/20 $ & 2275       & 125        & 65.77\% & 80.61\% & 85.63\% & 88.07\% \\ \hline
$1/40 $ & 9563       & 251        & 79.64\% & 89.56\% & 92.51\% & 93.88\% \\ \hline
$1/80 $ & 39220      & 502        & 88.78\% & 94.59\% & 96.19\% & 96.91\% \\ \hline
$1/160$ & 158851     & 1005       & 94.09\% & 97.24\% & 98.07\% & 98.45\% \\ \hline
	\end{tabular}
\end{table}

\begin{table}[h!]\centering
	\caption{Ratio $\frac{\text{WG-CG dofs}}{\text{Full-WG dofs}}$ for $P_k$ Elements}
	\label{tab:dof_compare_2}
	\begin{tabular}{c|c|c|c|c|c|c}
		\hline
$H$&CG Element     &WG Element   &$P_1$     &$P_2$     &$P_3$     &$P_4$      \\ \hline
$1/10 $ & 514        & 62         & 17.33\% & 27.27\% & 35.54\% & 42.26\% \\ \hline
$1/20 $ & 2275       & 125        & 12.68\% & 23.02\% & 31.71\% & 38.80\% \\ \hline
$1/40 $ & 9563       & 251        & 10.47\% & 21.00\% & 29.88\% & 37.15\% \\ \hline
$1/80 $ & 39220      & 502        & 9.39\% & 20.01\% & 28.99\% & 36.34\% \\ \hline
$1/160$ & 158851     & 1005       & 8.86\% & 19.53\% & 28.56\% & 35.95\% \\ \hline
	\end{tabular}
\end{table}

\subsection{Observation}\label{subsec:observation}

Across all tests, the numerical results in Tables
\ref{tab:test1_k1}--\ref{tab:test2_k2} confirm the theoretical convergence orders:
$\mathcal O(h^k)$ in the energy norm and $\mathcal O(h^{k+1})$ in $L^2(\Omega)$. The convergence rate for energy norm seems a bit higher constantly, though not proved here, we believe it is consistent with previous results, see \cite{wang2018supercloseness} for a super-convergence results of energy norm.
Moreover, Table~\ref{tab:dof_compare_1} shows a comparable dofs for CG and WG-CG methods when the mesh is fine enough.
Table~\ref{tab:dof_compare_2} highlights that restricting weak Galerkin unknowns to the boundary layer substantially reduces the total dofs relative to a fully nonconforming discretization, while maintaining optimal accuracy.

\section*{Declarations}

\subsection*{Conflict of interest}
The authors declare that they have no conflict of interest.

\subsection*{Data availability}
The data and code generated during the current study are available from the authors upon reasonable request.

\printbibliography

\end{document}